\documentclass[reqno,letter,11pt]{article}

\usepackage{epsf}
\usepackage{color}

\usepackage[latin1]{inputenc}
\usepackage[english]{babel}
\usepackage{amsmath,amssymb,amsfonts,amsthm,enumerate}
\usepackage{mathrsfs}

\definecolor{refkeybis}{gray}{.65}
\definecolor{labelkeybis}{gray}{.65}
{\makeatletter
\def\SK@refcolor{\color{refkeybis}}%
\def\SK@labelcolor{\color{labelkeybis}}}

\numberwithin{equation}{section}

\newtheorem{theorem}{Theorem}[section]
\newtheorem{lemma}[theorem]{Lemma}

\newtheorem{remark}[theorem]{Remark}
\newtheorem{proposition}[theorem]{Proposition}
\newtheorem{corollary}[theorem]{Corollary}

\newcommand{\N}{\mathbb{N}}
\newcommand{\R}{\mathbb{R}}

\renewcommand{\S}{\mathbb{S}}
\newcommand{\Haus}[1]{{\mathscr H}^{#1}}

\newcommand{\<}{\langle}
\renewcommand{\>}{\rangle}
\renewcommand{\a}{\alpha}

\renewcommand{\d}{\delta}

\newcommand{\e}{\varepsilon}

\renewcommand{\l}{\lambda}

\newcommand{\var}{\varphi}

\renewcommand{\t}{\tau}
\renewcommand{\O}{\Omega}

\renewcommand{\i}{\infty}
\newcommand{\p}{\partial}

\newcommand{\vol}{\operatorname{vol}}

\newcommand{\dist}{\operatorname{dist}}

\newcommand{\co}{\operatorname{co}}

\newcommand{\Cut}{{\rm Cut}}

\newcommand{\cexp}{{\hbox{\rm $c$-Exp}}}
\newcommand{\cexpb}{{\hbox{\rm $\bar c$-Exp}}}
\newcommand{\tcexp}{{\hbox{\rm $\tilde c$-Exp}}}

\newcommand{\cCut}{{\hbox{\rm $c$-Cut}}}
\newcommand{\cCutb}{{\hbox{\rm $\bar c$-Cut}}}
\newcommand{\domc}{M^*}
\newcommand{\domcb}{\bar M^*}
\newcommand{\cl}{{\rm cl}}
\newcommand{\cd}{{\cdot}}
\newcommand{\cdd}{{\cdot \cdot}}
\newtheorem{assumption}[theorem]{Assumption}

\title{Regularity of optimal transport maps\\ on  multiple products of spheres
\thanks{The authors are 
grateful to
the Institute for Pure and Applied Mathematics at UCLA and the Institute for Advanced Study in Princeton
for their generous hospitality during various stages of this work.
AF is partially supported by NSF grant DMS-0969962.
RJM is supported in part by
NSERC grants 217006-08 and 
NSF grant DMS-0354729.
YHK is supported partly by 
NSF grant DMS-0635607
through the membership at Institute for Advanced Study at Princeton NJ, and also in part by NSERC grant 371642-09. Any opinions, findings
and conclusions or recommendations expressed in this material are those of authors and do not
 reflect the views of either the Natural Sciences and Engineering
Research Council of Canada (NSERC) or the United States National Science Foundation (NSF).
\hfill \copyright 2010 by the authors.
}}
\date{\today}
\author{Alessio Figalli\thanks{Department of Mathematics, The University of Texas at Austin, Austin TX, USA. {\tt figalli@math.utexas.edu}},
Young-Heon Kim\thanks{ Department of Mathematics, University of
British Columbia, Vancouver BC, Canada. {\tt yhkim@math.ubc.ca}
} \ and Robert J. McCann\thanks{Department of Mathematics,
University of Toronto, Toronto ON, Canada. {\tt
mccann@math.toronto.edu}}}

\begin{document}

\maketitle
\begin{abstract}
This article addresses regularity of optimal transport maps for
cost$=$``squared distance'' on Riemannian
manifolds that are products of arbitrarily many round spheres
with arbitrary sizes and dimensions. Such manifolds are known
to be non-negatively cross-curved \cite{KM2}. Under boundedness and non-vanishing assumptions
on the  transfered source and target densities we show that optimal maps stay away from the cut-locus (where the cost exhibits singularity), and obtain injectivity and continuity of optimal maps.  This together with the result of Liu, Trudinger and Wang \cite{LTW} also implies higher regularity ($C^{1, \a}/C^\infty$) of optimal maps  for more smooth ($C^\a/C^\infty$) densities.
These are  the first global regularity results which we are aware of concerning optimal maps on
non-flat Riemannian manifolds which possess some vanishing sectional curvatures.
Moreover, such product manifolds have potential relevance
in statistics (see \cite{S}) and in statistical mechanics
(where the state of a system consisting of many spins is classically modeled
 by a point in the phase space obtained by taking many products of spheres).
For the proof we apply  and extend the method
developed in \cite{FKM}, where we showed injectivity and continuity of optimal maps on domains
in $\R^n$ for smooth non-negatively cross-curved cost.
 The major obstacle in the present paper is to deal with the non-trivial cut-locus and the presence of flat directions.
\end{abstract}

\tableofcontents

\section{Introduction}

Let $M$ and $\bar M$ be $n$-dimensional
complete Riemannian manifolds, and let  $\mu=\rho \vol_{M}$ and $\nu=\bar\rho \vol_{\bar M}$
be two probability measures whose densities $\rho$ and $\bar \rho$ are bounded away from zero and infinity.
Given a cost function $c:M\times \bar M \to \R$, the optimal transport problem with cost $c(x,y)$
consist in finding a transport map $T:M\to \bar M$ which sends $\mu$ onto $\nu$ and minimizes the transportation cost
$$
\int_{M} c(x,T(x))\,d\mu(x).
$$
 As shown by McCann \cite{M} extending the result of Brenier \cite{B} on $\R^n$, if  $M=\bar M$ and $c=\dist^2/2$ then the optimal transport map  (or simply optimal map) exists and is
unique. More generally, the same result holds if the cost is semiconcave and satisfies the twist
condition in Assumption~\ref{A:twist}, see \cite{Levin99,MTW,FF,F}.

The optimal map  $T$ is uniquely characterized by the relation
$T(x)\in\partial^c\phi(x)$, where $\phi$ is a $c$-convex function (called \emph{potential}) and
$\partial^c\phi$ denotes its $c$-subdifferential (see Section \ref{S:notation} for the definitions).
Furthermore, the fact that $\rho$ and $\bar \rho$ are bounded away from zero and infinity ensures the existence of a constant $\l>0$ such that the following Monge-Amp\`ere type equation holds:
\begin{align*}
\l\,|\O|\leq |\p^c\phi(\O)| \leq \frac{1}{\l}\,|\O|\qquad \forall \, \O\subset M\text{ Borel},
\end{align*}
where $\p^c\phi(\O)=\cup_{x\in\O}\p^c\phi(x)$. (See for instance \cite[Lemma 3.1]{FKM}.)

The aim of this paper is to investigate the regularity issue of optimal maps  when $M=\bar M$ are multiple
product of spheres, i.e., $M=\bar M=\S^{n_1}_{r_1} \times \ldots \times \S^{n_k}_{r_k}$,
and $c(x,y)=f(\dist(x,y))$ for some function $f$,  including the case $f(t)=t^2/2$ of distance squared cost. For $k=1$ and  $f(t)=t^2/2$, smoothness of optimal maps
has been proved  by Loeper \cite{L2}. However, if $k>1$ the structure of the cut-locus  (the singular set of the cost function)
becomes more complicated,  and due to the product structure, the manifold has both flat and positively curved directions, thus  making the regularity issue much more delicate. Especially, the powerful H\"older regularity estimate of Loeper \cite{L} (see also \cite{Liu09}) as well as the a priori estimates of Ma, Trudinger and Wang \cite{MTW}, which are successfully applied to positively curved manifolds as in \cite{L2, KM2, LV, FR, DG, FRV-reg}, are not available any more in our setting. Our main results (Theorem~\ref{T:stay-away} and Corollary~\ref{C:regularity}) give the first global regularity results which  we are aware of concerning optimal maps on non-flat Riemannian manifolds which allow vanishing sectional curvature. For completely flat manifolds (with $c=dist^2/2$) the regularity of optimal maps is known as it reduces to the regularity theory of the classical Monge-Amp\'ere equation \cite{D, caff-loc, caffC1a, C, U, C2, Co, D2, gutierrez}.

 To describe our result more precisely,
first recall that in \cite{MTW}
Ma, Trudinger and Wang discovered condition {\bf (A3)} on the cost function,  whose weaker
variant {\bf (A3w)} \cite{TW}
turned out to be both necessary \cite{L} and sufficient \cite{TW}
for regularity when the solution $\phi$
is known to be {\em strictly} $c$-convex and the cost function is smooth.
When $M=\bar M=\S^{n}_{r}$, the particular structure of the cut-locus (for every point $x$, its cut-locus $\Cut (x)$ consists of its antipodal point)
allowed Delano\"e and Loeper \cite{DL} to deduce
that  optimal maps \textit{stay away} from cut-locus, namely,  $\p^c\phi(x)\cap \Cut(x)=\emptyset$ for all $x \in M$
; see   \cite{L2,DG, KM1, KimMcCannAppendices}  for alternate approaches.
Loeper \cite{L2} combined this observation with the fact that $c=\dist^2/2$ satisfies
{\bf (A3)} to show regularity of optimal maps;
for a simpler approach to continuity, see  \cite{KM1, KimMcCannAppendices}.  His result has been extended to variety of positively curved manifolds including the complex projective space \cite{KM2} and perturbation of the real projective space \cite{LV} and of sphere \cite{FR, DG, FRV-reg}, all of where {\bf (A3)} holds thus the strong H\"older regularity estimate of \cite{L} as well as the a priori estimate of \cite{MTW} applies.  Note that {\bf (A3)} (resp. {\bf (A3w)}) forces the sectional curvature to be positive (resp. nonnegative) \cite{L}, though the converse does not hold \cite{K}. 

On multiple products of spheres, taking $c=\dist^2/2$ leads to two main issues: first, only a degenerate strengthening of the
weak Ma-Trudinger-Wang condition holds (the so-called \textit{non-negative cross-curvature} condition in \cite{KM1, KM2}),
which although stronger than {\bf (A3w)} is not as useful as
{\bf (A3)} for proving regularity  due to lack of powerful estimates; (neither non-negative cross-curvature nor {\bf (A3)}
implies the other,  though either one separately implies {\bf (A3w)}).
Moreover, the cut-locus now has a non-trivial structure, which makes it much more difficult to
understand whether the stay-away property holds.
In \cite{FKM} we showed strict $c$-convexity and $C^1$ regularity of $\phi$,  or equivalently, injectivity and continuity of $T$,   when the cost
is smooth and non-negative cross-curvature holds.
Hence the only question left is whether $\p^c\phi$ avoids the cut-locus or not.

In this paper we answer this question positively:
by taking advantage of the fact that the cut-locus is given by  the union of certain sub-products of spheres
we prove in Theorem~\ref{T:stay-away}  the stay-away property that $\p^c\phi(x)\cap \Cut(x)=\emptyset$ for all $x \in M$. By compactness, these two sets are separated by a uniform distance that is dependent  on $\lambda$, but independent of the particular choice of $\phi$ and $x$; see Corollary~\ref{C:uniform stay-away}.  Once stay-away property is shown, one can localize the argument of \cite{FKM}
to obtain injectivity and continuity of the optimal map; then higher regularity  ($C^{2,\a}/C^\infty$) of $\phi$, thus $C^{1,\a}/C^\infty$-regularity of $T$, follows from
\cite{LTW} when the densities are smooth  ($C^\a / C^\infty$); see Corollary~\ref{C:regularity}. 

The multiple products of spheres is a model case for more general manifolds on which  the cost $c$ satisfies the necessary conditions \cite{L, FRV-reg} for regularity of optimal transport maps. The method we develop in this paper demonstrates one approach to handling complex singularities of the cost, especially the stay-away property of optimal maps.  Moreover, a general Alexandrov type estimate (Lemma~\ref{L:left Alex}) is obtained which has applications  beyond the products of spheres.

Our regularity result has potential relevance to statistics and
statistical mechanics. 
 For instance, recently T. Sei applied optimal transport theory for $c=dist^2$ to directional
statistics on the sphere. In his main result \cite[Theorem 1]{S}, he needed the
optimal map not to touch the cut-locus. Now, our stay-away property on
multiple products of  spheres  $M$ (Theorem~\ref{T:stay-away}) states that all optimal maps, obtained by transporting
 densities bounded away from zero and infinity onto each other, satisfy
this assumption. Hence, this provides a large family of $c$-convex
potentials that  could be used to create log-concave likelihood functions
as in \cite[Subsection 3.2]{S}, extending his theory to multiple products of spheres. Namely, as a direct consequence of \cite{KM2, FigalliKimMcCann-econ09p, S}, on multiple products of  spheres  a convex combination $\phi=\sum_{i=1}^k s_i \phi_i$, $s_i \ge 0, \sum s_i =1$, of $c$-convex functions $\phi_i$ is again $c$-convex,  thus a crucial requirement  in Sei's theory is satisfied. If  each $\phi_i$ is the $c$-potential of an optimal map between densities bounded away from zero and infinity,  by Theorem~\ref{T:stay-away} one sees  $\partial^c \phi_i$ stay away from  the cut-locus. One then can show that $\partial^c \phi$ also avoids the cut-locus, thus applying \cite[Theorem 1]{S} one obtains the log-concave  Jacobian inequality for this convex combination. To see this, for example, observe that  in the product of spheres the domain of exponential map is convex and  $\partial^c \phi$ satisfies $ \partial^c \phi(x) = \exp_x \partial \phi (x) = \exp_x \big{[}\sum_i \partial \phi_i (x) \big{]})$ for $x \in M$ (see Lemma~\ref{L:DASM for c-sub}). Since each $\partial^c \phi_i$ stays away from the  cut-locus, $\partial \phi_i (x)$ belongs to the domain of exponential map, so does $\partial \phi(x)$, showing $\partial^c \phi(x) \cap \Cut(x) = \emptyset$.

Concerning statistical mechanics, let us recall that the state of a spin system
is classically modeled as a point in the phase space  $M$ obtained by taking many products of spheres.
In such contexts, optimal transport may provide a useful change of variables.
More precisely, if $\mu$ and $\nu$ are two smooth densities and $T$ denotes the optimal transport
map from $\mu$ to $\nu$, then
$$
\int G(y) \,d\nu(y)=\int G(T(x))\,d\mu(x),
$$
for all bounded measurable functions $G:M \to \R$.
Then, if $\mu$ is a ``nice'' measure for which many statistical quantities are easily computable,
one may hope to exploit some qualitative/quantitative
properties of $T$ in order to estimate the integral
$\int G(y) \,d\nu(y)$ by studying $\int G(T(x))\,d\mu(x)$. 
We expect that regularity of optimal maps may play a crucial role in this direction.  For instance, in Euclidean spaces this is already the case, as Caffarelli \cite{caffGauss} used regularity of optimal maps to show that
suitable monotonicity and log-concavity properties of the densities imply
monotonicity and  contraction properties for  the optimal map, from which correlation and momentum inequalities may be deduced. 
\\

{\bf Organization of the paper:} 
Section~\ref{S:notation} sets up the notation and assumptions used
throughout the paper. In Section~\ref{S:prelim}, a few useful
preliminary results regarding convex sets and $c$-convex functions
are listed. Section~\ref{S:left Alex} is devoted to an Alexandrov
type inequality which is one of the main tools in the proof of our
main theorem. Until Section~\ref{S:left Alex}, we present the
theory under rather general assumptions. However, from
Section~\ref{S:setting spheres} we restrict to the multiple
products of spheres. In Section~\ref{S:setting spheres} we state
our main result about the stay-away property
of optimal maps, and give a sketch of the proof. Moreover we
explain how one can deduce regularity of optimal maps combining
this theorem with the results in \cite{FKM} and \cite{LTW}.
Finally, the details of the proof of the stay-away property are
given in Section~\ref{S:proof}.\\

{\bf Acknowledgement:} The authors are pleased to thank Neil Trudinger,
Tom Spencer, and C\'edric Villani for useful discussions.

\section{Notation and assumptions}\label{S:notation}
In this section and the next we recall notation and results which will be useful
in the sequel. Many of these results originated in or were inspired by
the work of Ma, Trudinger, Wang \cite{MTW} and Loeper \cite{L}.
Though the present paper mainly concerns the Riemannian distance squared cost $c=\dist^2/2$
on the product of round spheres, we will present our work in a rather general framework.
It requires only a small additional effort and may prove useful for further
development and applications of the theory.

Let $M$, $\bar M$ be $n$-dimensional complete Riemannian
manifolds, and let $c(x, \bar x)$ denote a cost function $c: M \times
\bar M \to \R$.
We will assume through the whole paper that $c$ is semiconcave in both variables, i.e., in  coordinate charts it can be
written as the sum of a concave and a smooth function. Let us remark that
since $\dist^2(x,y)$ is semiconcave on $M\times M$ (see for example \cite[Appendix B]{FF}), the above assumption is satisfied for instance by any cost function of the form
$f(\dist(x,y))$ on $M\times M$, with $f :\R\to \R$ smooth, even, and strongly convex
(meaning $f(d) = f(-d)$ and $f''(d)>0$ for all $d \ge 0$). Here and in the sequel
we use smooth as a synonymous of $C^\infty$ (though $C^4$ would be enough for all our purposes).

As for $\bar x$ and $\bar M$, we use the ``bar'' notation to specify
the second variable of the cost function. Also as a notation we use $\bar c
(\bar x, x) := c(x, \bar x)$. We denote by $D_x$ and $D_{\bar x}$ the
differentials with respect to the $x$ and $\bar x$ variable respectively.
(For instance, $D_x D_{\bar x} c(x_0,\bar x_0)$ denotes the
mixed partial derivative of $c$ at $(x_0,\bar x_0)$.) Let $\cCut (\bar x)$ denote the
\emph{$c$-cut-locus} of $\bar x \in \bar M$, that is
\begin{align*}
\cCut (\bar x ) := \{ x \in M \ | \ \hbox{$c$ is not smooth in a neighborhood of $(x, \bar x)$} \},
\end{align*}
and let $M(\bar x)$ denote the \emph{$c$-injectivity locus} $M \setminus \cCut(\bar x)$. Define $\cCutb (x)$, $\bar M (x)$ similarly.
These sets are open.

\begin{assumption}[\bf twist]\label{A:twist}
For each $(x, \bar x) \in M\times \bar M$, the maps $-D_xc
(x, \cdot) : \bar M (x) \to T^*_x M$ and $-D_{\bar x}c (\cdot, \bar x)
: M (\bar x) \to T^*_{\bar x} \bar M$ are smooth embeddings (thus
injective).
\end{assumption}
We remark that the above hypothesis from Levin \cite{Levin99} is equivalent to
condition {\bf (A1)} in \cite{MTW,L,KM1}, which together with the semiconcavity of the cost
ensures existence and uniqueness of optimal maps when
the source measure is absolutely continuous with respect to the volume measure
(see for instance \cite{Levin99,FF,F} or
\cite[Chapter 10]{V}).

The \emph{domain of the $c$-exponential} $\domcb(\bar x)$ in
$T^*_{\bar x } \bar M$ is defined as the image of $M(\bar x)$
under the map $-D_{\bar x}c(x, \cdot)$, i.e.,
\begin{align*}
\domcb(\bar x):= -D_{\bar x} c( M(\bar x),  \bar x) \subset T_{\bar
x}^*\bar M.
\end{align*}
Define $\domc(x)$ similarly.

As in  \cite{MTW, L},
we define the \emph{$c$-exponential maps} $\cexp_x : \domc(x) \subset T^*_x M
\to \bar M$ and $\cexpb_{\bar x}: \domcb (\bar x) \subset
T^*_{\bar x} \bar M \to M $ as the inverse maps of $-D_xc(x, \cdot)$
and $- D_{\bar x} c(\cdot, \bar x)$ respectively, i.e.,
\begin{align*}
 p=- D_x c(x, \cexp_x p) \hbox{ for $p \in \domc(x)$}, \qquad \bar p= - D_{\bar x} c(\cexpb_{\bar x} \bar p , \bar x) \hbox{ for $\bar p \in \domcb(\bar x)$}.
\end{align*}

Given a set $X$, we denote by $\cl (X)$ its closure.
 Define the subdifferential of a semiconvex funciton $\a : M \to \R$ at $x \in M$ by
\begin{align*}
\partial \a (x) := \{ p \in T^*_x M \ |& \
\a(\exp_x v)-\a(x) \ge \langle p, v \rangle + o(|v|_x) \quad {\rm as}\ v \to 0\ {\rm in}\ T_x M \}
\end{align*}
(This is  non-empty at every point.)
Here $\langle \cdot, \cdot \rangle$ denotes the paring of covectors and vectors.
\begin{assumption}\label{A:c-exp} For each $(x, \bar x) \in M\times \bar M$ the map $\cexp_x$ (resp. $\cexpb_{\bar x}$) extends to a smooth map from
$\cl (\domc(x))$ (resp. $\cl (\domcb(\bar x))$) \emph{onto} $\bar M$ (resp. $M$).  If we abuse notation to use $\cexp_x$, $\cexpb_{\bar x}$ to denote these extensions, then they satisfy  
\begin{align*}
\cexp_x p  & = \cexp_x \Big{(} \partial_x \big{(}-c(x, \cexp_x p) \big{)} \Big{)}, \forall p \in \cl(\domc(x)); \\
\cexpb_{\bar x} \bar p  & = \cexpb_{\bar x} \Big{(} \partial_{\bar x}\big{(}- c(\cexp_{\bar x} \bar p, \bar x) \big{)} \Big{)}, \forall \bar p \in \cl(\domcb(\bar x)).
\end{align*}
Here, $\partial_x$, $\partial_{\bar x}$ denote the subdifferentials with respect to the variables $x$, $\bar x$, respectively. 
\end{assumption}

Note that the above assumptions hold for instance when $M=\bar M$ and
$c=\dist^2/2$ (so that $\cexp_x$ coincides with the Riemannian
exponential map $\exp_x$). However, the following three assumptions
are much more restrictive, and not true for $c=\dist^2/2$ in
general \cite{MTW, L, KM1, LV}. They are  all crucial in this paper.

\begin{assumption}[\bf convexity of domains of $c$-exponentials]\label{A:domain convex}
For each $(x, \bar x) \in M\times \bar M$ the domains $\domc (x)$, $\domcb (\bar x)$ are convex.
\end{assumption}

As shown in \cite{FRV-reg}, the above assumption is necessary for continuity of optimal 
transport maps when the cost function is given by the squared distance.

A \emph{$c$-segment $\{\bar x(t)\}_{0 \le t \le 1}$ with respect to $x$} is the $c$-exponential image of a line segment in  $\cl ( \domc(x))$, i.e.,
\begin{align*}
\bar x(t):= \cexp_x ( (1-t) p_0 + t p_1 ) , \qquad \hbox{ for some
$p_0, p_1 \in T^*_x M$ }.
\end{align*}
Define similarly a $\bar c$-segment $\{ x(t)\}_{0 \le t \le 1}$ with
respect to $\bar x$. The notions of $c$- and $\bar c$-segments, due to Ma, Trudinger and Wang,
induce a  natural extension of the notion of convexity on sets in $M$, $\bar M$
called \textit{$c$-convexity} in \cite{MTW}. Let $U \subset
M$, $\bar x \in \bar M$. The set $U$ is said to be \emph{$\bar
c$-convex with respect to $\bar x$} if any two points in $U$ are
connected by a $\bar c$-segment with respect to $\bar x$ entirely
contained inside $U$. Similarly we define $ c$-convex sets in
$\bar M$. It is helpful to notice that $c$, $\bar c$-convex sets
(with respect to $x$, $\bar x$, respectively) are images of convex
sets under $\cexp_x$, $\cexpb_{\bar x}$, respectively.

Regarding $c$, $\bar c$-segments, here comes a key assumption in this paper:
\begin{assumption}[{\bf convex DASM}]\label{A:convex DASM} For every $(x, \bar x) \in M \times \bar M$,
let $\{\bar x(t)\}_{0 \le t \le 1}$, $\{ x(t) \}_{0\le t \le 1}$ be $c$, $\bar c$-segments with respect to $x$, $\bar x$, respectively. Define the functions
\begin{align*}
m_t (\cdot) := -c(\cdot, \bar x(t) ) + c(x, \bar x(t)), \qquad
\bar m_t (\cdot) := -c(x(t), \cdot ) + c(x(t),\cdot ), \qquad 0\le
t\le 1.
\end{align*}
Then
\begin{align}\label{E:convex DASM}
m_t  \le  (1-t) m_0 + t \, m_1, \qquad \bar m_t \le (1-t) \bar m_0 + t \, \bar m_1 \qquad 0 \le t \le 1.
\end{align}
\end{assumption}
\noindent
{When, instead of
\eqref{E:convex DASM}, only $m_t \le \max [ m_0, m_1]$ and $\bar
m_t \le \max [ \bar m_0, \bar m_1]$ are required, this property played
a key role in the work of Loeper \cite{L}. In \cite{KM1} we called it
{\em Loeper's maximum principle} ({\bf DASM}),
the acronym ({\bf DASM}) standing for ``Double Mountain Above Sliding Mountain'', a mnemonic which
describes how the graphs of the functions $m_{t}$, $\bar m_{t}$ behave as $t$ is varied.
For convenience we use this acronym in various places in the present paper.
The stronger property ({\bf convex  DASM}) was proved in \cite{KM2} to be a consequence of
the so-called nonnegative cross-curvature condition on the cost $c$.

We will also need a strict version of Loeper's maximum principle ({\bf DASM}):}
\begin{assumption}[{\bf DASM$^+$}]\label{A:DASM+}
With the same notation as in Assumption \ref{A:convex DASM},
$$
m_t(y) \le \max [ m_0(y), m_1(y)]\quad \forall \,y \in M, \qquad
\bar m_t(\bar y) \le \max [ \bar m_0(\bar y), \bar m_1(\bar y)] \quad \forall\,\bar y \in \bar M.
$$
Moreover,  when the $\bar c$-(resp. $c$-)segument in the definition of $m_t$ (resp. $\bar m_t$) is nonconstant, the equality holds if and only if $y=x$ (resp. $\bar y=\bar x$).
\end{assumption}
Assumptions \ref{A:convex DASM} and \ref{A:DASM+} correspond to a ``global''
version of the non-negative cross curvature assumption and of the {\bf (A3)} condition of the cost function $c$, respectively: see \cite{KM1} and
\cite{MTW} for the definition of nonnegative cross curvature and {\bf (A3)}, respectively.
Although the equivalence between ({\bf convex  DASM}) and non-negative cross curvature (resp. ({\bf DASM$^+$}) and {\bf (A3)}) is not known in general,
it holds true for  the squared distance cost function on a Riemannian manifold,
as shown in \cite{FV,FRV-reg}. Moreover,   Loeper's maximum principle ({\bf DASM}) is a necessary condition for regularity: this is originally shown \cite{L} on domains in $\R^n$ and later extended to the manifold case \cite{FRV-reg}.

Given two functions $\phi:M\to \R$ and $\bar \phi:\bar M\to \R$, we say that they are \emph{$c$-convex}  and dual with respect to each other if
\begin{align}\label{E:c-convex}
\phi(x) &= \sup_{\bar x \in \bar M} \{- c(x, \bar x) -\bar \phi (\bar x)\},\\\nonumber
\bar \phi(\bar x) &= \sup_{x \in M} \{- c(x, \bar x) -\phi (x)\}=\sup_{x \in M} \{- \bar c(\bar x, x) -\phi (x)\}.
\end{align}
Since by assumption $c$ is semiconcave, both functions above are semiconvex (see for instance \cite[Appendix A]{FF}).
This implies in particular that their subdifferentials, $\partial \phi (x)$, $\p\bar \phi (\bar x)$ 
are non-empty at every point.

We define the \emph{$c$-subdifferential} $\partial^c \phi$ at a point $x$ as follows:
\begin{align}\label{subdiff}
\partial^c \phi (x) := \{ \bar x \in \bar M \ | \ \phi (y) - \phi (x) \ge  - c(y, \bar x) + c(x, \bar x), \ \forall\, y \in M \}.
\end{align}
Analogously, we define $\partial^{\bar c} \bar \phi$ at every point $\bar x$. (Recall that $\bar c$ denotes
the function defined as $\bar c(\bar x, x) := c(x, \bar x)$.)
The following {well-known} reciprocity holds: 
\begin{lemma}[\bf Reciprocity]\label{L:reci} For $c$-convex functions $\phi$, $\bar \phi$ dual to each other as in \eqref{E:c-convex},
\begin{align}\label{reci}
\bar x \in \partial^c \phi (x)  \Longleftrightarrow 
\phi(x) + \bar\phi(\bar x) = -c(x, \bar x)
\Longleftrightarrow x \in \partial^{\bar c} \bar \phi (\bar x).
\end{align}
\end{lemma}
\begin{proof}
Suppose $\bar x \in \partial^c \phi (x)$. Then, by rearranging the inequality in \eqref{subdiff} we get
$$
-\phi (x) \ge c(x, \bar x) +\sup_{y \in M} \{ -c(y, \bar x) -\phi (y)\},
$$
and the supremum on the right hand side is exactly $\bar \phi (\bar
x)$. On the other hand, from the definition of $\phi$ and $\bar \phi$ we have
$$
\bar \phi (\bar y) + c(x, \bar y) \ge - \phi (x)
\qquad \forall\,\bar y \in \bar M,
$$
so that combining these two inequalities leads to $\phi(x) + \bar\phi(\bar x) = -c(x, \bar x)$,  and  $x \in \partial^{\bar c} \bar \phi (\bar x)$. The opposite implication follows by symmetry.
\end{proof}

Loeper \cite{L} deduced the following fundamental relation to be a consequence
of his maximum principle ({\bf DASM}).

\begin{lemma}[\bf Loeper's maximum principle ({\bf DASM})]\label{L:DASM for c-sub}
Let Assumptions~\ref{A:twist}, \ref{A:c-exp} and \ref{A:domain convex} hold. Suppose {\em Loeper's maximum principle ({\bf DASM})} holds. Let  $\phi$, $\bar \phi$ be $c$-convex functions dual to each other as in \eqref{E:c-convex}.
Then  for all $x \in M, \bar x \in \bar M$, 
$$
\cexp_{x}(\partial \phi (x) ) = \partial^c \phi (x), \qquad
\cexpb_{\bar x}(\partial \bar \phi (\bar x)) = \partial^{\bar c} \bar \phi (\bar x).
$$
\end{lemma}
\begin{proof}
  
The inclusions $\cexp_{x}(\partial \phi (x) ) \subset \partial^c \phi (x)$, $\cexpb_{\bar x}(\partial \bar \phi (\bar x)) \subset \partial^{\bar c} \bar \phi (\bar x)$ follow from the convexity of $\partial \phi (x)$ and the definition of {\em Loeper's maximum principle} ({\bf DASM}).
The other inclusions hold in general without Loeper's maximum principle.  Details can be found in \cite{L}.
 \end{proof}
\noindent In the following, we refer the conclusion of this lemma also as Loeper's maximum principle ({\bf DASM}).

For a set $\Omega \subset M$, the image $\partial^c \phi (\Omega)$ is defined as
\begin{align*}
\partial^c \phi (\Omega) := \bigcup_{ x \in \Omega} \partial^c \phi (x).
\end{align*}
For a $c$-convex function $\phi$ and an open set $U \in M$ with
$x_0 \in U$, we define the set $[\partial^c \phi (U)]_{x_0} \subset
\bar M$ as
\begin{align*}
 [\partial^c \phi (U)]_{x_0} := \{ \bar x \in \bar M \ | \ \phi(x)-\phi(x_0) \ge -c(x, \bar x) + c(x_0, \bar x) \hbox{ for all $x \in \partial U$} \}.
\end{align*}
Trivially, $\partial^c \phi (x_0) \subset [\partial^c \phi (U)]_{x_0} $.
This definition is justified by the following lemma, which is also very useful in later discussions.
\begin{lemma}\label{L:DASM for c-sub of a set}
Let Assumptions~\ref{A:twist}, \ref{A:c-exp} and \ref{A:domain convex} hold.  Suppose Loeper's maximum principle ({\bf DASM}) holds. Let $\phi$ be a $c$-convex function on $M$. Let $U \subset M$ be an open set, and let $x_0 \in U$. Then
\begin{itemize}
\item[(1)]  $[\partial^c \phi (U)]_{x_0}$ is $ c$-convex with
respect to $x_0$; \item[(2)]  $[\partial^c \phi (U)]_{x_0} \subset
\partial^c\phi (U)$; \item[(3)] If $U \to \{x_0\}$, then both
$\partial^c\phi (U)$,  $[\partial^c \phi (U)]_{x_0} \to \partial^c
\phi(x_0)$.
\end{itemize}
\end{lemma}
\begin{proof}
Assertion (1) follows directly from the definitions of Loeper's maximum principle ({\bf DASM})
and of the set $[\partial^c \phi (U)]_{x_0}$.

To prove Assertion (2), fix $\bar x \in [\partial^c \phi (U)]_{x_0}$, and move first the graph of the function
$-c(\cdot,\bar x)$ down so that it lies below $\phi$ inside $U$, and then lift it up until it touches the graph of $\phi$ inside $\cl(U)$.
Thanks to the assumption $\bar x \in [\partial^c \phi (U)]_{x_0}$ there exists at least one touching point $x'$ which belongs to $U$
(indeed, if there is a touching point on $\p U$, then $x_0$ is another touching point),
and Lemma~\ref{L:DASM for c-sub} ensures that $\bar x \in \partial^c \phi(x')$.

For (3), the
convergence $\partial ^c \phi (U) \to \partial^c \phi (x_0)$
follows by continuity, and $[\partial^c \phi (U)]_{x_0} \to
\partial^c \phi(x_0)$ comes then from (2).
\end{proof}

For $\bar x \in \bar M$, let $S(\bar x)$ be the \emph{contact set}
\begin{align*}
S(\bar x) : = \{ x \in M  \ |  \ \bar x \in \partial^c \phi (x) \} = \partial^{\bar c} \bar \phi (\bar x).
\end{align*}
(The last identity follows from reciprocity, see
Lemma~\ref{L:reci}.) For any $x_0 \in S(\bar x)$ one can write
\begin{align*}
S(\bar x) = \{ x \in M  \ | \ \phi (x ) -\phi (x_0) = - c( x, \bar
x) +  c(x_0, \bar x) \}.
\end{align*}
A set $Z$ in $M$ is called a \emph{$c$-section} of $\phi$ with respect to $\bar x$ if
there is  $\lambda_{\bar x} \in \R$ such that
\begin{align*}
Z:= \{ z \in M \ | \ \phi(z) \le -c(z, \bar x) + \lambda_{\bar x}
\ \}.
\end{align*}

The following simple observation is very useful for studying regularity of $c$-convex functions.
It was originally made (implicitly) in \cite{FKM} and independently by Liu \cite{Liu09}.
\begin{lemma}[\bf $c$-convex $c$-sections]\label{L:c-convex c-section}
Let Assumptions~\ref{A:twist}, \ref{A:c-exp} and \ref{A:domain convex} hold. Suppose Loeper's maximum principle ({\bf DASM}) holds. Let $\phi$ be a  $c$-convex function on $M$, and fix $\bar x \in \bar M$. Every $c$-section $Z$ of $\phi$ with respect to $\bar x$ is $\bar c$-convex with respect to $\bar x$.
\end{lemma}
\begin{proof}
 This follows from the definition of $c$-convex functions and Loeper's maximum principle ({\bf DASM}).
\end{proof}

Given Borel sets $V\subset M$ and $\bar V\subset \bar M$, we denote by $|V|$ and $|\bar V|$
their volume (computed with respect to the given Riemannian metric on $M$ and $\bar M$, respectively).
The following is our last assumption. As we already remarked in the introduction,
it is satisfied whenever $\phi$ is the potential associated to an optimal transport map
and the densities are both bounded away from zero and infinity.

\begin{assumption}[\bf bounds on $c$-Monge-Amp\`ere measure of $\phi$]\label{A:Monge Ampere}
There exists $\lambda>0$ such that
\begin{align*}
\lambda |\Omega|
\le |\partial^c \phi (\Omega)|
\le \frac{1}{\lambda} |\Omega| \hbox{ \ for all Borel set $\Omega \subset M$}.
\end{align*}
We sometimes abbreviate this condition on $\phi$ simply by writing
$|\partial^c \phi| \in [\lambda, \frac{1}{\lambda}]$.
\end{assumption}

\section{Preliminary results}\label{S:prelim}

 In this section, we list some preliminary results we require later. The first subsection deals with general convex sets and the second subsection considers the properties of the cost function under suitable assumptions.
\subsection{Convex sets}
We first list two properties of convex sets that will be useful later.
\begin{lemma}[\bf John's lemma]\label{L:John}
For a compact convex set $S \subset \R^n$, 
there exists an affine transformation $ L: \mathbb{R}^n \to \mathbb{R}^n$ such that
$B_1 \subset L^{-1}(S) \subset B_{n}$. Here, $B_1$ and $B_n$ denote the ball of radius $1$ and $n$, respectively, centered at $0$. 
\end{lemma}
\begin{proof}
See \cite{John}.
\end{proof}

\begin{lemma}
\label{lemma:orthogonal sections}
Let $S$ be a convex set in $\R^n = \R^{n'} \times \R^{n''}$, and denote by
$\pi', \pi''$ the canonical projections onto $\R^{n'}$ and
$\R^{n''} $, respectively.
Let $S'$ be a slice orthogonal to the second component, that is
$$
S' = (\pi'')^{-1}(\bar x'')\cap  S \qquad\text{for some }\bar x'' \in \pi''(S).
$$
Then there exists a constant $C(n)$, depending only on $n=n'+n''$,
such that
$$
C(n) \,|S| \ge \Haus{n'}(S') \Haus{n''}(\pi'' (S)),
$$
where $\Haus{d}$ denotes the $d$-dimensional Hausdorff measure.
\end{lemma}
\begin{proof}
See \cite[Lemma 7.8]{FKM}.
\end{proof}
 The following lemma is important in the last step (Section~\ref{S:final argument})
 of the proof of the main theorem.
\begin{lemma}\label{lemma:mutiple orthogonal sections}
Let $X=X^1 \times \ldots \times X^k$, with $X^i=\R^{n_i}$,
$i=1,\ldots, k$, and write a point $x \in X$ as $x=(x^1,
\ldots, x^k)$, $x^i \in X^i$. For each $i=1,\ldots, k$, let $U^i$ be a subset of
$X^i$, and let $s_i=(s_i^1,\ldots,s_i^k) \in X$ with $s_i^i \in
U^i$. Define $S_i \subset X$ as
$$
S_i:=\{s_i^1\} \times\ldots \times\{s_i^{i-1}\} \times U_i\times\{s_i^{i+1}\}
\times\ldots \times \{s_i^k\},
$$
and consider the convex hull $\co(S_1, \ldots, S_k)$ of the sets $S_1,
\ldots, S_k$. Then there exists a constant $C(n,k)$, depending only on $n:=n_1 + \ldots + n_k$ and $k$, such that
$$
C(n, k)\, |\co(S_1, \ldots, S_k)| \ge \Pi_{i=1}^k
\Haus{n_i} (S^i).
$$
\end{lemma}
\begin{proof}
First consider the barycenter $b$ of the set $\{s_1, \ldots,
s_k\}$, that is
$$
b := \frac{1}{k} (s_1 + \ldots + s_k).
$$
We will construct sets $S_i^b$ each of which  contains $b$ and has Hausdorff measure comparable with $S_i$. In addition, these sets
are mutually orthogonal. We will finish the proof by considering
the volume of the convex hull of these sets $S_1^b, \ldots,
S_k^b$.

For each $i$, let $b_i$ be the barycenter of the set $\{s_1,
\ldots, s_k\} \setminus \{ s_i\}$, i.e.,
$$
b_i := \frac{1}{k-1} (s_1 + \ldots+s_{i-1} + s_{i+1} + \ldots+ s_k).
$$
Consider the cone $\co(b_i, S_i) \subset \co(S_1, \ldots, S_k)$
and let $S_i^b$ be the intersection
$$
S_i^b := \co(b_i, S_i) \cap \{ x \in X \ | \ x^j =b^j \hbox{ for
$j\ne i$} \}.
$$
Note that $b \in S_i^b$ and these sets $S_1^b, \ldots, S_k^b$ are
mutually orthogonal, in the sense that, for each $x \in S_i^b$ and $y
\in S_j^b$ with $i\ne j$, it holds $(x-b) \cdot (y-b) =0$. Now,
consider the convex hull $\co(S_1^b, \ldots, S_k^b) \subset
\co(S^1, \ldots, S^k)$. The previous orthogonality implies
$$
C(n)\, |\co(S_1^b, \ldots, S_k^b)| \ge \Pi_{i=1}^k
\Haus{n_i}(S_i^b).
$$
for some constant $C(n)$ depending only on $n_1+
\ldots+ n_k$. (This inequality is obtained for instance by
iteratively applying Lemma~\ref{lemma:orthogonal sections}.) To
conclude the proof simply observe that $b=\frac{1}{k} s_i +
\frac{k-1}{k}b_i$, and so
$$
\Haus{n_i}(S_i^b)  \geq \frac{1}{k^{n_i}}\Haus{n_i}(S_i).
$$
\end{proof}

\subsection{Coordinate change}\label{SS:coordinate}
In this subsection we briefly recall the coordinate change introduced in \cite[Section 4.1]{FKM} that transforms
$c$-convex functions into convex functions under the condition ({\bf
convex DASM}), referring to \cite[Section 4.1]{FKM} for more details. Throughout this subsection we let Assumptions~\ref{A:twist}, \ref{A:c-exp} and \ref{A:domain convex} hold.

Let $\bar y_0$ be an arbitrary point in $\bar M$.
Then the map $x \in M(\bar y_0) \mapsto q \in T^*_{\bar y_0} \bar M$
given by $q (x)= - D_{\bar x}c(x, \bar y_0)$ is an embedding thanks to Assumption~\ref{A:twist}.  Recall that
$\bar M^* (\bar y_0) \subset T^*_{\bar y_0} \bar M$ denotes the image of this
map, that this map is by definition the inverse $c$-exponential
map $(\cexp_{\bar y_0})^{-1}$, and the $c$-exponential map is a diffeomorphism up to the boundary of $\bar M^* (\bar y_0)$ (see Assumption~\ref{A:c-exp}). Denote
\begin{align*}
\tilde c (q, \bar x) := c(x(q), \bar x) - c(x(q), \bar y_0).
\end{align*}
Then the $c$-convex function $\phi$ is transformed to a $\tilde c$-convex function $\varphi$ defined as
\begin{align*}
\varphi (q) := \phi (x(q) ) + c (x(q), \bar y_0).
\end{align*}
If Loeper's maximum principle ({\bf DASM}) holds, then Lemma~\ref{L:c-convex c-section} shows that
$\tilde c$-sections of $\varphi$ are convex. This property was observed independently
by Liu \cite{Liu09}, who used it to derive an optimal H\"older exponent for optimal
maps under the strict condition {\bf (A3)} on the cost, sharpening the H\"older continuity result
of Loeper \cite{L}.
Furthermore, if ({\bf convex DASM}) holds then $-\tilde c(q, \bar x)$ is
convex in $q$ for any $\bar x \in \bar M$, which then implies
convexity of $\varphi$ in $q$ (see \cite[Theorem 4.3]{FKM} for more details).
One can easily check that $\bar c$-segments with respect to $\bar
x$ are transformed via this coordinate change to $\tilde{\bar
c}$-segments with respect to $\bar x$, and $c$-segments with
respect to $x(q)$ are transformed to $\tilde c$-segments with
respect to $q$. Therefore, Loeper's maximum principle ({\bf DASM}) or ({\bf convex DASM}) for $c$
implies the same for $\tilde c$.

\subsubsection{Relation between cotangent vectors in
two different coordinates}
Here we give an explicit relation
between covectors in the new coordinate variable $q$ (as introduced above) and
the original coordinate variable $x$. 
Fix arbitrary $\bar  y_0 \in
M$, $x_0 \in M(\bar y_0)$, and let  $q_0 = - D_{\bar x} c(x_0,
\bar y_0) \in T^*_{\bar y_0 } \bar M$. For each $\bar z \in \bar M(x_0)$,
consider the maps
\begin{align}\label{eq:eta p relation}
&\bar z \mapsto \eta (\bar z) := -D_x c(x_0, \bar z) \in T^*_{x_0
} M\\\nonumber & \bar z \mapsto p(\bar z) := - D_q \tilde c(q_0,
\bar z) \in  T^*_{q_0} (T^*_{\bar y_0} \bar M).
\end{align}
 where $\tilde c (q, \bar x) := c(x(q), \bar x) - c(x(q), \bar y_0)$  and the variables  $x$ and $q$ are related as $q (x)= - D_{\bar x}c(x, \bar y_0)$.
Denote by $M^*(x_0)$, $\tilde M^*(q_0)$ the embedding of $\bar M(x_0)$
under the mappings $\bar z \mapsto \eta(\bar z)$, $\bar z \mapsto
p (\bar z)$, respectively.  These sets are related by an affine map as we see in the following lemma. In particular, from Assumption~\ref{A:domain convex} both sets are convex in $T^*_{x_0 } M$, $T^*_{q_0} (T^*_{\bar y_0}
M)$, respectively.

\begin{lemma}\label{L:eta shift}
Let Assumptions~\ref{A:twist} and \ref{A:c-exp} hold. Let $\eta(p)$ denote the map from $\tilde M^*(q_0)$ to $M^*(x_0)$ that associates $p(\bar z)$ to $\eta(\bar z)$ as in the relation \eqref{eq:eta p relation},
and let $\eta_0 =  -D_xc(x_0, \bar y_0) \in M^*(x_0) \subset T^*_{x_0 } M$.
 Fix local coordinates. Then for all $p \in \tilde M^*(q_0)$,   $\eta(p)= (\eta(p)^1, \cdots, \eta(p)^n)$ is given as
$$\eta (p)^i = p^j \,\bigl(-D_{x^i} D_{\bar x^j} c (x_0, \bar y_0)\bigr) + \eta_0^i.$$
This formula allows the affine function $p\mapsto \eta(p)$
to be extended to a global map $\eta: T^*_{q_0}(T^*_{\bar y_0} \bar M) \to T^*_{x_0} M$.
\end{lemma}
\begin{proof}

Observe that
\begin{align*}
 \eta^i &= - D_{x^i}c ( x_0, \bar z)\\
 &= D_{x^i}\big{|}_{x=x_0} [ -c(x, \bar z) + c(x, \bar y_0) - c(x, \bar y_0) ] \\
&= -D_{q^j} \big{|}_{q=q_0}[ c (x(q), \bar z) - c(x(q), \bar y_0)]
(D_{x^i} \big{|}_{x=x_0} q^j) + \eta_0^i\\
&=p^j\,(D_{x^i} \big{|}_{x=x_0} q^j) + \eta_0^i.
\end{align*}
From the relation
$$
q^j= -D_{\bar x^j} c(x(q), \bar y_0 )
$$
we see that
$$
D_{x^i} \big{|}_{x=x_0} q = -D_{x^i} D_{\bar x^j} c (x_0, \bar y_0),
$$
and the assertion follows.
\end{proof}

\subsubsection{An estimate on the first derivatives of $c$}
In Section~\ref{S:regular component} we will use the following
simple estimate.
\begin{lemma}\label{L:c estimate}
Given  convex sets $\Omega, \Lambda \subset \R^n$, assume that the function $(q,y) \in \Omega \times \Lambda \mapsto c(q,y) \in \R$ is smooth. Then for all $q, \tilde q \in \Omega$ and
$y \in \Lambda$ we have
\begin{align}\label{slope compare}
|-D_q c(q,y)+D_qc(\tilde q,y)|
&\le C | q - \tilde q|\, |D_q c(\tilde q,y)|,
\end{align}
where the constant $C$ depends only on $\|c\|_{C^3 (\Omega \times \Lambda)}$ and $\|(D^2_{qy}c)^{-1}\|_{L^\i(\Omega \times \Lambda)}$.
\end{lemma}
\begin{proof}
 See \cite[Lemma 7.7]{FKM}.
\end{proof}

\section{ An Alexandrov estimate: upper bound}\label{S:left Alex}
In this section we show a key Alexandrov type estimate
\eqref{eq:left Alex} which bounds from above the size of
a $c$-section, say $Z_h$, by its height $h$. (An estimate that
bounds the size of the $c$-section $Z_h$ either from above or
below by its `height' $h$ is called \emph{Alexandrov type}.) This
result  is  of its own interest, especially because it is proven under rather general assumptions, and does not rely
on the special structure of products of spheres. In later
sections, a companion inequality showing the lower bound will be
obtained for a special choice of a $c$-section in the particular case of
products of spheres, see Theorem~\ref{T:right Alex}.

The key point in the estimate below is that the term $\frac{\max_{x \in Z_h} |\det
(-D_x D_{\bar x} c (x, \bar x_0))|}{\min_{x \in Z_h}|\det
(-D_xD_{\bar x} c(x, \bar x_0))|}$ appearing in \eqref{eq:left Alex} can be made as close to $1$ as desired, provided
one can ensure that the section $Z_h$ converges to a point as $h \to 0$. This fact will play a crucial role in the proof of Theorem~\ref{T:stay-away}

\begin{lemma}[\bf Alexandrov upper bound]\label{L:left Alex}
Let $M$, $\bar M$ be complete $n$-dimensional Riemannian
manifolds. Suppose the cost $c: M\times \bar M \to \R$ satisfies  Assumptions~\ref{A:twist}, \ref{A:c-exp}, \ref{A:domain convex} and \ref{A:convex DASM}
({\bf convex DASM}). Let $\phi$ be a $c$-convex function on $M$ and
assume $0 < \lambda \le |\partial^c \phi |$ for a fixed $\lambda
\in \R$. Fix $(x_0, \bar x_0) \in M \times \bar M$ such that $\bar
x_0 \in \partial^c \phi (x_0)$, and for $h>0$ consider the
$c$-section $Z_h$ defined as
\begin{align*}
Z_h := \{ x \in M \ | \  \phi (x) -\phi (x_0 ) \le -c(x, \bar x_0)
+ c(x_0, \bar x_0) + h \}.
\end{align*}
Assume that $-c(\cdot, \bar x_0)$ is smooth on $ Z_h$, so that the function  $\cexpb_{\bar x_0}^{-1}$ is   defined and smooth  on $Z_h$, or equivalently  $ Z_h \subset M(\bar x_0)$.
Then the following inequality holds:
\begin{align}\label{eq:left Alex}
\lambda |Z_h|^2  \le C(n) \biggl[\frac{\max_{x \in Z_h} |\det
(-D_x D_{\bar x} c (x, \bar x_0))|}{\min_{x \in Z_h}|\det
(-D_xD_{\bar x} c(x, \bar x_0))|}\biggr]^2 \Big[\sup_{x \in Z_h}
\sup_{p' \in \domc (x)} |d\cexp_{x}\big{|}_{p=p'}|\Big{]}  \,  h^n
\end{align}
 with the constant $C(n) = (4n)^n |B_1|^2 $.
\end{lemma}

\begin{remark}\label{R:DbarDc nonsingular}
{\rm In the statement of Lemma~\ref{L:left Alex} and its proof,  it is important to notice that by the assumption  $-c(\cdot, \bar x_0)$ is smooth on  $\cl Z_h$ and Assumptions~\ref{A:twist} and \ref{A:c-exp}, the derivatives of $\cexpb_{\bar x_0}$ and its inverse (on $Z_h$), i.e., $-D_x D_{\bar x} c (x, \bar x_0)$, $x \in Z_h$,  are all nonsingular.}
\end{remark}
 \begin{remark}\label{R:c-exp contraction}{\rm
For $c=\dist^2/2$, Loeper's maximum principle ({\bf DASM}) (and so also ({\bf convex DASM}))
implies that $M=\bar M$ has nonnegative sectional curvature  (see \cite{L}).
Therefore in this case $\cexp_{y}$ is a contraction, that is
\begin{align*}
 \sup_{p' \in \domc (x)} |d\cexp_{x}\big{|}_{p=p'}|
\le 1.
\end{align*}
We do not know if this contraction property holds for general
non-negatively cross-curved cost functions.}
\end{remark}

\begin{proof}
For globally smooth cost functions (on the products of two bounded
domains) a similar result was proved in  \cite[Proposition 7.3]{FKM}. In the present case
where the cost function has singularities, the previous proof does
not work any more  and we require the following subtle argument. 

As in \cite{FKM}, we will follow the strategy developed in \cite{caff-loc} by using renormalization techniques, but only after a suitable change
of coordinates. Consider the coordinate change $ x \in Z_h \mapsto q =
-D_{\bar x} c(x(q), \bar x_0 ) \in \bar W_h \subset \bar M^*(\bar x_0)\subset
T_{\bar x_0}^* \bar M$, i.e.,
$x= \cexp_{\bar x_0} q$ and $Z_h=\cexp_{\bar x_0} (\bar W_h)$, and let
\begin{align*}
 m_{\bar x} (\cdot) := -c (\cdot, \bar x) + c (\cdot, \bar x_0).
\end{align*}
As explained in Section~\ref{SS:coordinate}, in these new coordinates the functions
\begin{align*}
 \hbox {$q \mapsto m_{\bar x} (x(q))\quad$ and
 $\quad q\mapsto \varphi(q)=\phi(x(q)) + c(x(q), \bar x_0)$}
\end{align*}
are convex. Moreover the set $\bar W_h$ is convex,
as
\begin{align*}
 \bar W_h = \{ q \in T_{\bar x_0}^* \bar M \ | \ \varphi(q) -\varphi(q_0) \le h \},
\end{align*}
where $q_0$ is the point corresponding to $x_0$ in the new
coordinates, i.e.,  $\cexp_{\bar x_0} q_0 = x_0$. It is also
important to notice that   $\bar x_0 \in \partial^c \phi (x_0)$
implies $\varphi(q) -\varphi(q_0) \ge 0$. We now use Lemma~\ref{L:John} to find an affine map $A: T_{\bar x_0}^* \bar M \simeq \R^n \mapsto T_{\bar x_0}^* \bar M \simeq \R^n$ such
that $A(\hat{W}_h)=\bar W_h$, with $B_1 \subset \hat{W}_h \subset B_n$.
Denote $q_b = A(0)$ and $x_b = \cexp_{\bar x_0} q_b$. Define the
renormalized function $\hat \varphi (\hat q) := \varphi (A \hat q)$
for each $\hat q \in \hat W_h$, and denote $\frac{1}{2} \bar W_h :=
A(\frac{1}{2} \hat W_h)$ and $\frac{1}{2} Z_h := \cexpb_{\bar x_0}
(\frac{1}{2} \bar W_h)$, where $\frac{1}{2} \hat W_h$ denotes the dilation of $\hat W_h$ by a factor $1/2$ with respect to the origin.  This $1/2$-dilation (or any factor in $(0,1)$ works) will be important  in this proof.

Consider the reciprocal expression
\begin{align*}
 \partial^c \phi \Big(\frac{1}{2} Z_h\Big) = \cexp_{x_b} (- D_x c(x_b, \bar x_0) + \mathcal{V})
\end{align*}
where
\begin{align*}
\mathcal{V} := \Big\{ D_x m_{\bar x} (x_b)  \ | \ \bar x \in
\partial^c \phi \Big(\frac{1}{2} Z_h \Big)\Big\} \subset
T^*_{x_b}M.
\end{align*}
Here $D_x  m_{\bar x}$ denotes the differential when $m_{\bar x}$
is differentiable, otherwise it means an arbitrary covector in the
subdifferential $\partial m_{\bar x}(x)$. Notice that $- D_x c(x_b,
\bar x_0) + \mathcal{V} \subset \cl (\domc(x_b))$, and thus
\begin{align}\label{E:partial c of Z less V}
\Big|\partial^c \phi \Big(\frac{1}{2} Z_h\Big) \Big| \le
\Big(\sup_{p' \in \domc (x_b)} |d\cexp_{x_b}\big{|}_{p=p'}|\Big)
|\mathcal{V}|.
\end{align}
Now, the left-hand side is bounded from below as
\begin{align}\nonumber
\Big|\partial^c \phi \Big(\frac{1}{2} Z_h \Big) \Big| & \ge
\lambda \Big| \frac{1}{2} Z_h \Big| \qquad \hbox{(by the
assumption $|\partial^c \phi| \ge \lambda$)} \\\nonumber & \ge
\lambda \Big{[}\min_{w\in W_h} |\det (d\cexpb_{\bar x_0}
\big{|}_{q=w})|\Big{]} \Big(\frac{1}{2}\Big)^n |W_h| \qquad
\hbox{(by $\frac{1}{2} Z_h = \cexpb_{\bar x_0}
(\frac{1}{2} W_h)$) }\\\nonumber & \ge \lambda \frac{\min_{w\in W_h}
|\det (d\cexpb_{\bar x_0} \big{|}_{q=w})|}{\max_{w\in W_h} |\det
(d\cexpb_{\bar x_0} \big{|}_{q=w})|} \Big(\frac{1}{2}\Big)^n
|Z_h|\\\label{E:left by Z} & \ge  \lambda \frac{\min_{x \in Z_h}
|\det (-D_x D_{\bar x} c (x, \bar x_0))|}{\max_{x \in Z_h}|\det
(-D_xD_{\bar x} c(x, \bar x_0))|} \Big(\frac{1}{2}\Big)^n |Z_h|
\qquad \hbox{(by $D_{\bar x} c(\cdot , \bar x_0)  = \cexpb_{\bar
x_0}^{-1}$)}.
\end{align}
In the following we will bound $|\mathcal{V}|$ from above by
\begin{align*}\frac{\max_{x \in Z_h} |\det (-D_x D_{\bar x}
c (x, \bar x_0))|}{\min_{x \in Z_h}|\det (-D_xD_{\bar x} c(x, \bar
x_0))|}\frac{h^n}{|Z_h|},
\end{align*}
which will finish the proof;  here the dilation $\frac{1}{2} Z_h$ plays a crucial role (see \eqref{E:m * bound}). Fix $\bar x \in
\partial^c \phi (\frac{1}{2} Z_h)$, and let $\hat q_{\bar x} \in \frac{1}{2} \hat W_h$ such that $\bar x \in \partial^{\hat
c} \hat \varphi (\hat q_{\bar x})$. Here, the cost function $ \hat
c$ is the modified cost function accordingly with the coordinate
changes:
\begin{align*}
 \hat c (\hat q, \bar y) := c (\cexpb_{\bar x_0}(A\hat q), \bar y) - c (\cexpb_{\bar x_0}(A\hat q), \bar x_0).
\end{align*}
Consider the  function
\begin{align}\label{E:m and m *}
\hat m_{\bar x} (\hat q) := m_{\bar x} (\cexpb_{\bar x_0} (A\hat
q))= -\hat c (\hat q, \bar x)  .
\end{align}
Then
\begin{align*}
\hat m_{\bar x}(\hat q) - \hat m_{\bar x} (\hat q_{\bar x}) + \hat
\varphi(\hat q_{\bar x})   \le \hat \varphi (\hat q)
 \qquad \hbox{for  $\hat q \in \hat W_h$}.
\end{align*}
We observe that $\hat m_{\bar x}(\cdot) - \hat m_{\bar x} (\hat
q_{\bar x})$ is a convex function on $\hat W_h$ which vanishes at $\hat
q_{\bar x} \in \frac{1}{2}\hat W_h$, and $\hat m_{\bar x}(\cdot) -
\hat m_{\bar x} (\hat q_{\bar x}) \leq h$ on $\partial \hat W_h$.
Since $B_1 \subset \hat W_h \subset B_{n}$ this easily gives $\hat
m_{\bar x}(0) - \hat m_{\bar x} (\hat q_{\bar x}) \geq -h$, which by convexity implies
\begin{align}\label{E:m * bound}
|D_{\hat q}\hat m_{\bar x} (0)| \le 2 h .
\end{align}
To get information on  $D_x m_{\bar x}$, observe that from
\eqref{E:m and m *}
\begin{align}\label{E:m and m * diff}
 (d\cexpb_{\bar x_0} \big{|}_{q=q_b} A)^*D_x m_{\bar x}(x_b) = D_{\hat q}\hat m_{\bar x}
 (0),
\end{align}
where $(d\cexpb_{\bar x_0} \big{|}_{q=q_b} A)^* : T^*_{x_b}M \mapsto T^*_0 (T_{\bar x_0}^* \bar M)$ is the dual map of the derivative map $d\cexpb_{\bar x_0} A : T_0 (T_{\bar x_0}^* \bar M) \mapsto T_{x_b} M$.
Here we abuse the notation and $A$ denotes both the affine map and its derivative.
Moreover we use the canonical identification $T^*_0 (T_{\bar x_0}^* \bar M) \approx T_0 (T_{\bar x_0}^* \bar M) \approx T_{\bar x_0}^* \bar M$.
Hence \eqref{E:m * bound} and \eqref{E:m and m * diff} imply  the key inclusion
\begin{align*}
\mathcal{V} \subset \big(d\cexp_{\bar x_0}
\big{|}_{q=q_b}^*\big)^{-1} (A^*)^{-1} B_{2h},
\end{align*}
so that
\begin{align*}
|\mathcal{V}| & \le \big|\det (d\cexpb_{\bar x_0} \big{|}_{q=q_b}^*)^{-1}\big| \big|\det (A^*)^{-1}\big| |B_1|2^n h^n\\
 & = |\det (d \cexpb_{\bar x_0}
\big{|}_{q=q_0})|^{-1} |\det A|^{-1} |B_1|2^n h^n\\
& \qquad \hbox{(by the identification between vectors and covectors)}\\
& \le  |\det (d \cexpb_{\bar x_0}
\big{|}_{q=q_0})|^{-1}  |B_1|^2 (2n)^n \frac{h^n}{|W_h|} \qquad \hbox{(by $|W_h| = |\det A||\hat W_h| \le |\det A| |B_1|n^n$)}  \\
& \le  \frac{\max_{w \in W_h} |\det (d\cexpb_{\bar x_0}\big{|}_{q=w})|}{|\det (d \cexpb_{\bar x_0}
\big{|}_{q=q_0})|} |B_1|^2 (2n)^n \frac{h^n}{|Z_h|}\\
& \qquad \hbox{(by $ |Z_h| \le \max_{w \in W_h} |\det (d\cexpb_{\bar x_0}\big{|}_{q=w})| |W_h|$ )}\\
& \le \frac{\max_{x \in Z_h} |\det (-D_x D_{\bar x} c (x, \bar
x_0)|)}{\min_{x \in Z_h}|\det (-D_xD_{\bar x} c(x, \bar x_0))|}
|B_1|^2 (2n)^n \frac{h^n}{|Z_h|} \qquad \hbox{(by $D_{\bar x} c(\cdot , \bar x_0)  = \cexpb_{\bar x_0}^{-1}$).}
\end{align*}
Together with \eqref{E:partial c of Z less V} and \eqref{E:left by Z}, this concludes the proof.
\end{proof}

\section{Stay-away property on multiple products of spheres}
\label{S:setting spheres} From now on we restrict our attention to
the case $M = \bar M = M^1 \times \ldots \times M^k$, where for each $i=1, \ldots, k$,
$M^i=\mathbb S_{r_i}^{n_i}$ is a round sphere of constant sectional curvature $r_i^{-2}$.
Though $M=\bar M$, we sometimes keep the bar notation
to emphasize the distinction between the source and the target
domain of the transportation. Let  $x = (x^1, \ldots, x^k)$ and
$\bar x = (\bar x^1, \ldots, \bar x^k)$ denote points in the
product $M^1 \times \ldots \times M^k$, with $x^i, \bar x^i \in
M^i$, $i=1, \ldots, k$. Assume that the transportation cost $c$ on
$M$ is the tensor product of the costs $c^i$ on each $M^i$,
defined as
\begin{align}\label{E:cost on product}
 c(x, \bar x) := \sum_{i=1}^k c^i (x^i, \bar x^i).
\end{align}
Assume moreover that each $c^i$ is of the form $f^i(\dist_i)$
($\dist_i$ being the distance on $M^i$) for some smooth
strongly convex even function $f^i: \R \to \R$,
normalized so that $f^i(0)=0$. (This normalization
assumption can be done with no loss of generality,
as one can always add an arbitrary constant to the cost function.)
Moreover we suppose that each $c^i$ satisfies Assumptions~\ref{A:twist}, \ref{A:c-exp},
\ref{A:domain convex}, \ref{A:convex DASM} and \ref{A:DASM+} in
Section~\ref{S:notation}. As shown in \cite{KM2},  under these assumptions
the tensor product cost $c$ also satisfies
Assumptions~\ref{A:twist}, \ref{A:c-exp},
\ref{A:domain convex} and \ref{A:convex DASM} (but not necessarily \ref{A:DASM+}). The reader should have in mind that our model
example is $c=\dist^2/2$, which as shown in \cite{KM2}
satisfies all the assumptions above. However we prefer to give a proof
of the result with general $f^i$ since this will not cost further effort in the proof,
and we believe it may be of interest for future applications.

Let us observe that for any point $\bar x$ we have $M(\bar x) = M^1 (\bar x^1) \times \ldots \times M^k(\bar x^k) $ and $\domcb (\bar x) = \domcb (\bar x^1) \times
\ldots \times \domcb (\bar x^k)$. Moreover, since that the distance squared
function on a round sphere is smooth except for antipodal pairs,
for each $x^i \in M^i$ we have $\Cut (x^i) = \{-x^i\}$, where
$-x^i$ denotes the antipodal point of $x^i$. (We also write $-x = (-x^1, \cdots, -x^k)$.)  This implies easily
that $\cCut(x) = \Cut(x)$, so that
$M^i(x^i)=M^i\setminus\{-x^i\}$ and $\cCut(x)$ is a union of
(totally geodesic) submanifolds, each of which is an embedding of a
product $M^{i_1} \times \ldots \times M^{i_l}$, $l \le k$.

The goal of the rest of the paper is to show a stay-away property
of optimal transport maps on products of spheres:
\begin{theorem}[\bf Stay-away from cut-locus]\label{T:stay-away}
 Let  $M = \bar M = M^1 \times \ldots \times M^k$, where for each $i=1, \ldots, k$,
$M^i=\mathbb S_{r_i}^{n_i}$ is a round sphere of constant sectional curvature $r_i^{-2}$.
Let $c$
be the cost given in \eqref{E:cost on product} with $c^i$ is of the form $f^i(\dist_i)$,
where $f^i: \R \to \R$ are smooth strongly convex even functions such that $f^i(0)=0$.
Assume further that each cost $c^i$
satisfies Assumptions~\ref{A:twist}, \ref{A:c-exp}, \ref{A:domain
convex}, \ref{A:convex DASM} and \ref{A:DASM+}, and let
$\phi$ be a $c$-convex function satisfying Assumption~\ref{A:Monge
Ampere}. Then
$$
\partial^c \phi(x) \cap \cCut(x)=\emptyset\qquad \forall\,x \in M.
$$
Equivalently, for every $\bar x \in \bar M$ the contact set $S(\bar x) =\partial^{\bar c} \bar \phi (\bar x)$ satisfies
\begin{align*}
 S(\bar x)  \cap \cCutb(\bar x)=\emptyset .
\end{align*}
\end{theorem}

Before sketching the proof of this result, let us first see its
consequences:

\begin{corollary}[\bf Uniformly stay-away from cut-locus]\label{C:uniform stay-away}
Use the notation and assumptions as in Theorem~\ref{T:stay-away}. There exists a positive constant $C$ depending only on $\lambda$ (see Assumption~\ref{A:Monge Ampere}) and $n_i$, $r_i$, $ f^i$, for $i=1, \cdots, k$, such that
\begin{align*}
\dist \big( \partial^c \phi(x) ,  \cCut(x) \big) \ge C \qquad \forall\,x \in M.
\end{align*} 
where $\dist$ denotes the Riemannian distance of $M$. 
\end{corollary}
\begin{proof}
The result follows by compactness. Indeed, suppose by contradiction there exists a sequence of $c$-convex functions $\phi_l$ satisfying Assumption~\ref{A:Monge Ampere}, and $x_l \in M$ such that 
 \begin{align*}
\dist \big(\partial^c \phi_l (x_l), \cCut(x_l)\big) \to 0 \qquad \hbox{ as $l \to \infty$}.
\end{align*}
Up to adding a constant, we can also assume that $\phi_l(x_l)=0$.
Then, since $M$ is compact and the functions $\phi_l$ are uniformly semiconvex (and so uniformly Lipschitz),
applying Arzel\`a-Ascoli's Theorem, up to a subsequence there exists a $c$-convex function $\phi_\infty$ and $x_\infty \in M$ such that $\phi_l  \to \phi_\infty$ uniformly and $x_l \to x_\infty$.
We now observe that also $\phi_\infty$ satisfies Assumption~\ref{A:Monge Ampere}
(see for instance \cite[Lemma 3.1]{FKM}).
Moreover, by the definition of $c$-subdifferential we easily obtain 
$$
y_l \in \partial^c\phi_l(x_l),\quad y_l\to y_\infty
\quad \Rightarrow \quad  y_\infty \in \partial^c\phi_\infty(x_\infty).
$$
This implies
$$
\dist \big(\partial^c \phi_\infty (x_\infty), \cCut(x_\infty)\big) = 0,
$$
which contradicts Theorem~\ref{T:stay-away}, and completes the proof. 
\end{proof}

\begin{corollary}[\bf Regularity of optimal maps]\label{C:regularity}
Let $M,\bar M,c$ be as in Theorem~\ref{T:stay-away}.
Assume that $\mu$ and $\nu$ are two probability measures absolutely
continuous with respect to the volume measure, and whose densities
are bounded away from zero and infinity. Then the unique optimal
map $T$ from $\mu$ to $\nu$ is injective and  continuous. 
Furthermore, if both densities are $ C^\a/C^{\infty}$, then $T$ is  $C^{1,\a}/C^{\infty}$.
\end{corollary}
 \begin{remark}\label{R:regularity}
{\rm The $C^{1,\a}$-regularity result ($C^{2,\a}$ for the potential $\phi$) in this corollary  is a direct consequence of the injectivity and continuity of $T$ applied to the theory of Liu,Trudinger and Wang \cite{LTW}. The higher regularity $C^\infty$ follows from Schauder estimates.}   
\end{remark}
\begin{proof}
We recall that, under the assumption that $\mu$ and $\nu$ have
densities bounded away from zero and infinity, there exists a
$c$-convex function $\phi$ such that $T(x)=\cexp_x(\nabla
\phi(x))$ a.e., and $\phi$ satisfies Assumption~\ref{A:Monge
Ampere} (see for instance \cite{MTW} or \cite[Lemma 3.1]{FKM}). Hence it suffices to prove
that $\phi$ is $C^{1}$ and strictly $c$-convex, in the sense that $S(\bar x)=\partial^{\bar c} \bar \phi (\bar x)$
is a singleton for every $\bar x \in \p^c\phi(M)=M$.

To this aim, we observe that once we know that $\phi$ is strictly $c$-convex, then we can localize the proof of the  $C^{1}$ regularity in \cite{FKM} 
to obtain the desired result. Thus we only need to show the
strict $c$-convexity of $\phi$.

Fix $\bar x \in M$. By Theorem \ref{T:stay-away} we know that
$S(\bar x) \subset M(\bar x)$, so that in a neighborhood of
$S(\bar x)$ we can consider the change of coordinates $x \mapsto
q= -\bar Dc(x,\bar x) \in T_{\bar x}^* \bar M$. As shown in
\cite[Theorem 4.3]{FKM}, thanks to Loeper's maximum principle ({\bf DASM}) the set $S(\bar x)$ is convex in
these coordinates.
Moreover, since now the cost is smooth in a
neighborhood of $S(\bar x)$, by \cite[Theorem 8.1 and Remark 8.2]{FKM} the  compact convex set $S(\bar
x)$ in the new coordinates has no exposed points on the support of $|\partial^c\phi|$.
Since in our case the support of $|\partial^c\phi|$ is the whole
$M$, the only possibility left is that $S(\bar x)$ is a singleton,
as desired.
\end{proof}

\begin{proof}[Sketch of the proof of Theorem \ref{T:stay-away}]
We prove this theorem by contradiction. Assume there exists a
point $\bar x_0$ such that the contact set $S(\bar x_0)$ intersects $\Cut(\bar x_0)$. First, we find a
\emph{cut-exposed} point $x_0$ in $S(\bar x_0) \cap \Cut(\bar
x_0)$. More precisely we split $M$ as $M^\cdot \times M^\cdd$ so
that $\bar x_0=(\bar x_0^\cdot,\bar x_0^\cdd)$,
$x_0=(x_0^\cdot,x_0^\cdd)$, where $x_0^\cdot=-\bar x_0^\cdot\in\Cut(\bar x_0^\cdot)$, $x_0^\cdd$ stays away from the
cut-locus of $\bar x_0^\cdd$, and $x_0^\cdd$ is an exposed point
in  the  set $S^{\cdot \cdot}= \{ y^{\cdot \cdot} \in M^{\cdot \cdot} \ | \ (-\bar x_0^{\cdot}, y^{\cdot \cdot}) \in S(\bar x_0)  \}$ (see Section~\ref{S:cut-exposed}).
Near $\bar x_0$, for $\e \in (0,1)$ and $\d\in[0, 1]$ we construct a family of points
$\bar x_{\e,\d}=(\bar x_\e^\cdot,\bar x_\d^\cdd)$ such that $d(\bar x_0,\bar x_{\e,\d})\approx \e+\d$, so that for $\delta$ small we have
$\bar x_{\e,\d}\in \bar M^*(x_0)$, or equivalently $x_0 \in
M^*(\bar x_{\e,\d})$.
By suitably choosing the point $\bar x_\d^\cdd$ in order to exploit the fact that $x_0^\cdd$ is an exposed point for  $S^\cdd$, we can ensure that, if $Z_{\e,\d, h}$ denotes
a section obtained by cutting the graph $\phi$ with $-c(\cdot,\bar x_{\e,\d})$ at height $h$ above $x_0$, then for any fixed $\e \in (0,1)$ we have $Z_{\e,\d, h}
\to \{x_0\}$ as $\d,\frac{h}{\d} \to 0$ (see Section~\ref{S:analysis near
cut-exposed}). In particular, for $\e>0$ fixed we have $Z_{\e,\d,h} \subset M^*(\bar x_{\e,\d})$ for $\d,\frac{h}{\d}$ small (equivalently, the function $c(\cdot,\bar x_{\e,\d})$ is smooth inside $Z_{\e,\d, h}$).
Now we take advantage of the choice of $\bar x_{\e}^\cdot$: on the sphere
$\S^{n}$ the function $-\dist^2(\cdot,\bar x)$ looks like a cone near the antipodal
point $-\bar x$, and if $\dist(\bar x,\bar x_\e) \approx \e$ then
the measure of a section obtained by cutting the graph
of $-\dist^2(\cdot,\bar x)$ with $-\dist^2(\cdot,\bar x_\e)$ at height $h$ above $-\bar x$
has measure $\approx
h^{n}/\e$ (see
Proposition~\ref{P:|Z h i|}).
\begin{figure}[h]
\centerline{\epsfysize=1.5truein \epsfbox{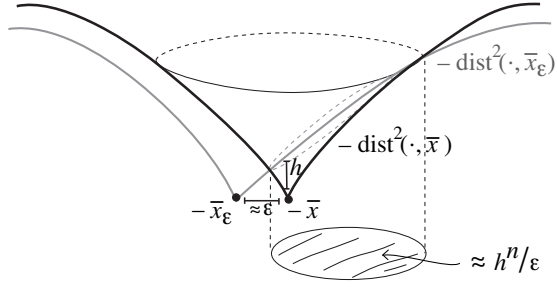}}
\caption{{\small On the sphere, the squared distance function from a point $\bar x$ looks like a cone near $-\bar x$. So, if $\dist(\bar x,\bar x_\e)=\dist(-\bar x,-\bar x_\e) \approx \e$, the section obtained by
cutting its graph
with $-\dist^2(\cdot,\bar x_\e)$ at height $h$ has measure $\approx h^{n}/\e$.}}
\end{figure}
In our case, since  $x_0^\cdot=-\bar x_0^\cdot$,
the function $\phi$  behaves as $-c(\cdot,\bar x_0)\approx -\dist(\cdot,\bar x_0)$ along $M^\cdot$ (see Lemma~\ref{L:phi on M i x 0}).
Hence by the argument above we have an improvement of a factor $1/\e$ in the measure of $Z_{\e,\d, h}$ (see Proposition~\ref{P:|Z h i|}),
which
allows to show the following Alexandrov type inequality:
\begin{align*}
h^{\dim M} \lesssim \e |Z_{\e,\d,h}| |\partial^c \phi
(Z_{\e,\d,h})| \qquad \hbox{for $\d$ and $\frac{h}{\d} $ sufficiently small,}
\end{align*}
where $\lesssim$ is independent of $\e,\d$ and $h$ (see Theorem~\ref{T:right Alex}).
Thanks to Assumption~\ref{A:Monge Ampere}, the above inequality implies
\begin{align}\label{eq:right Alex in sketch}
 h^{\dim M} \lesssim \frac{\e}{\lambda} |Z_{\e,\d,h}|^2 \qquad\hbox{for $\d$ and $\frac{h}{\d} $ sufficiently small}.
\end{align}
On the other hand, since $Z_{\e,\d,h} \subset M^*(\bar x_{\e,\d})$ for $\d$ and $\frac{h}{\d} $
small enough, we can
apply Lemma~\ref{L:left Alex} to $Z_{\e,\d,h}$ and have
\begin{align*}
\lambda |Z_{\e,\d,h}|^2  \le C(n) \bigg{[}\frac{\max_{x \in
Z_{\e,\d,h}} |\det (-D_x D_{\bar x} c (x, \bar
x_{\e,\d})|}{\min_{x \in Z_{\e,\d,h}}|\det (-D_xD_{\bar x} c(x,
\bar x_{\e,\d}))|}\bigg{]}^2 \Big{[}\sup_{x \in Z_{\e,\d,h}}
\sup_{p' \in \domc (x)} |d\cexp_{x}\big{|}_{p=p'}|\Big{]}  \,
h^{\dim M}.
\end{align*}
The convergence $Z_{\e,\d, h} \to \{x_0\}$ as $\d,\frac{h}{\d}  \to 0$ further
reduces this inequality to
\begin{align*}
\lambda |Z_{\e,\d,h}|^2 \lesssim h^{\dim M} \qquad \hbox{for $\d$
and $\frac{h}{\d}$ sufficiently small},
\end{align*}
which contradicts \eqref{eq:right Alex in sketch} as $\e \to 0$
and completes the proof.
\end{proof}

The rest of the paper is devoted to fleshing out the details of the above proof.

\section{Proof of Theorem \ref{T:stay-away} \bf (Stay-away from cut-locus)}
\label{S:proof}

\subsection{Cut-exposed points of contact
sets}\label{S:cut-exposed}
Assume by contradiction that there exists $\bar x_0 = (\bar
x_0^1, \ldots, \bar x_0^k) \in \bar M = M = M^1 \times \ldots
\times M^k$ such that $S(\bar x_0) \cap \cCut (\bar
x_0)\neq \emptyset$.
To prove Theorem~\ref{T:stay-away} a first step is to find a \emph{cut-exposed point}
of the contact set in the intersection with the cut-locus, which we define throughout the present section.

Let $y \in S(\bar x_0) \cap \cCut (\bar
x_0)$, and note that one of the components of $y =
(y^1, \ldots, y^k)$, say $y^j$, satisfies $y^j = -\bar x_0^j$. Moreover we
cannot have $y=-\bar x_0$. Indeed it is not difficult to see that,
if $\bar x_0 \in \partial^c\phi(-\bar x_0)$, then
$\partial^c\phi(-\bar x_0)=M$ (see for instance Lemma \ref{L:phi
on M i x 0}(1) below), which contradicts Assumption \ref{A:Monge
Ampere}.

Among all   points $y \in S(\bar x_0)\cap \cCut(\bar x_0)$, choose one such that the
number $a_0$ of its antipodal (or \emph{cut-locus}) components is maximal,
and denote the point by $y_0$. By rearranging the product $M^1
\times \ldots \times M^k$, we may write with out loss of
generality that
\begin{align}\label{a 0}
 y_0 = (-\bar x_0^1, \ldots, -\bar x_0^{a_0}, y_0^{a_0 +1} , \ldots, y_0 ^k), \qquad y_0^j \not\in \cCut(\bar x_0^j) \quad \forall\,j=a_0+1,\ldots,k.
\end{align}
For convenience, use the expression
\begin{align*}
 M^\cdot = \bar M^\cdot = M^1 \times \ldots \times M^{a_0}, &\qquad
M^{\cdot \cdot} = \bar M^{\cdot \cdot} = M^{a_0 +1} \times \ldots
\times M^k.
\end{align*}
The expressions $A^\cdot$, $A^{\cdot \cdot}$ will be used to denote things defined for elements in $M^\cdot$, $M^{\cdot \cdot}$, respectively. For example,
\begin{align*}
y^\cdot = (y^1, \ldots, y^{a_0}), &\qquad
y^{\cdot \cdot} = (y^{a_0+1}, \ldots, y^k),\\
c^\cdot (y^\cdot, \bar y^\cdot) = \sum_{i=1}^{a_0} c^i (y^i, \bar y^i), &\qquad c^{\cdot \cdot} (y^{\cdot \cdot}, \bar y^{\cdot \cdot}) = \sum_{i=a_0+1}^k c^i (y^i, \bar y^i).
\end{align*}
Consider the set
\begin{align*}
S^\cdd= \{ y^{\cdot \cdot} \in M^{\cdot \cdot} \ | \ (-\bar x_0^{\cdot}, y^{\cdot \cdot}) \in S(\bar x_0)  \}.
\end{align*}
 Notice that due to maximality of $a_0$, $S^\cdd \subset M^\cdd (\bar x_0^\cdd)$ and it is embedded to $\bar M^{\cdot \cdot *}(\bar x_0^{\cdot \cdot})$ through the map
$y^{\cdot \cdot} \mapsto q^\cdd(y^\cdd) =-D_{\bar x^{\cdot \cdot}}
c^{\cdot \cdot}( y^{\cdot \cdot}, \bar x_0^{\cdot \cdot})$. Observe that since $S^{\cdot \cdot}$ is compact, the
resulting set, say $\tilde S^{\cdot \cdot}$, is compact too. Moreover $\tilde S^{\cdot \cdot}$
is convex since it is the restriction of the convex set $\tilde
S(\bar x_0)$ to $M^{\cdot \cdot *}(\bar x_0^{\cdot \cdot})$, where
$\tilde S(\bar x_0)$ is the image of $S(\bar x_0)$ under the map
$x \mapsto -D_{\bar x} c(x, \bar x_0)$. (More precisely, this set
$\tilde S (\bar x_0)$ is defined as the closure of the image of
$S(\bar x_0) \setminus \cCut(\bar x_0)$.) 
This compact convexity
ensures the existence of an exposed point $q_0^{\cdot \cdot}$ for $\tilde S^{\cdot
\cdot}$, that is, there exists an affine
function $L$ on  $T^*_{\bar x_0^\cdd} \bar M^\cdd$ such
that
\begin{align}\label{eq:L}
 L (q_0^{\cdot \cdot}) = 0, \ \ \ \hbox{and}  \ \ \ L ( q^{\cdot \cdot}) < 0  \ \ \forall  q^{\cdot \cdot} \in \tilde S^{\cdot \cdot}\setminus\{ q_0^{\cdot \cdot}\}.
 \end{align}
(In case $\tilde S^{\cdot \cdot} = \{ q_0^{\cdot \cdot}\}$ let $L \equiv 0 $.) 
One should note that if $L$ is such an affine function, then $t L$ is also such an affine function for any $t>0 $.
Let $x_0 \in S(\bar x_0)$ be the corresponding point of $q_0^{\cdot \cdot}$ in $M$, that is,
\begin{align}\label{x 0}
 x_0 = (-\bar x_0^{\cdot}, x_0^{\cdot \cdot}),
\end{align}
where
\begin{align*}
q_0^{\cdot \cdot} =  -D_{\bar x^{\cdot \cdot}} c^{\cdot \cdot} (x_0^{\cdot \cdot}, \bar x_0^{\cdot \cdot}). 
\end{align*}
We call this point $x_0$ a \emph{cut-exposed} point of $S(\bar
x_0)$, since its components are either cut-locus type or exposed.
\begin{figure}[h]
\centerline{\epsfysize=1.5truein \epsfbox{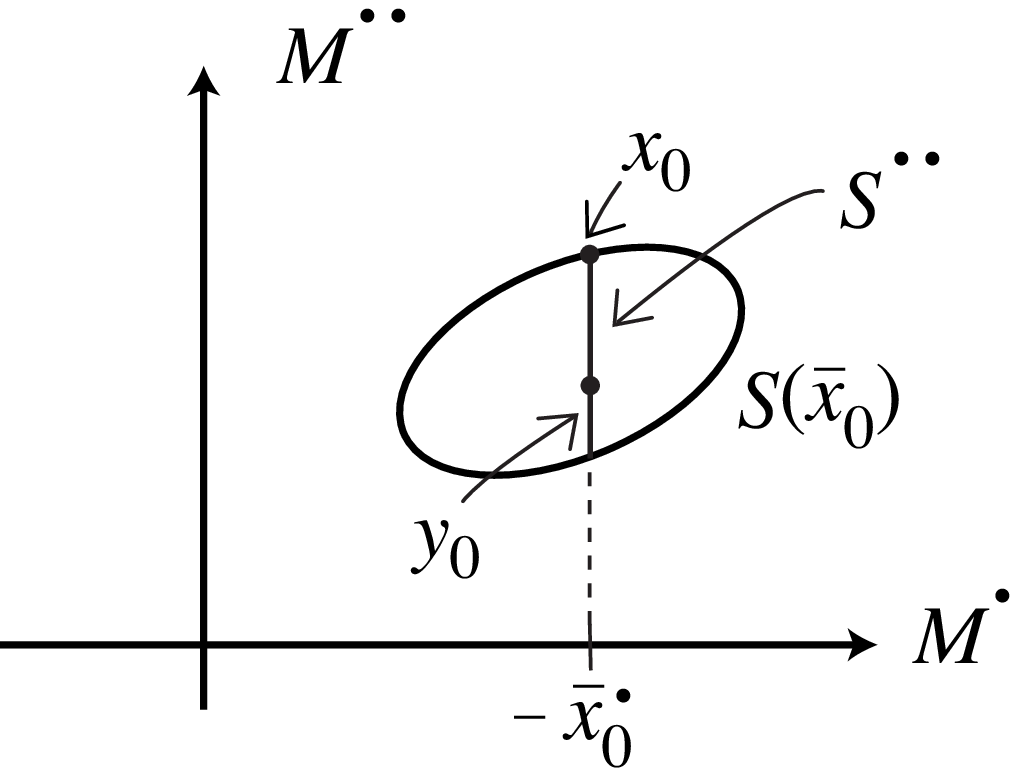}}
\caption{{\small Starting from a point $y_0 \in S(\bar x_0)$ such that the
number of its antipodal (or \emph{cut-locus}) components is maximal, we choose $x_0=(y_0^\cdot,x_0^{\cdot \cdot})=(-\bar x_0^\cdot,x_0^{\cdot \cdot}) \in M^\cdot \times M^{\cdot \cdot}=M$ so that
$x_0^{\cdot \cdot}$ is an exposed point for $S^{\cdot \cdot}$ (in some suitable system of coordinates).}}
\end{figure}

One can assume with a further rearrangement of the product
$M^{\cdot \cdot} = M^{a_0 +1} \times \ldots \times M^k$ that there
exists $b_0 \in \{a_0 , \ldots, k\} $ with the following two
properties:
\begin{enumerate} \item[1.] For each $i_1 \in \{ a_0 +1, \ldots, b_0\}$, there
exists $\bar y_{i_1} \in \partial^c \phi (x_0)$ with
\begin{align}\label{E:y j in cut}
\bar y_{i_1}^{i_1} & = -x_0^{i_1}.
\end{align}
\item[2.] For every $\bar y \in
\partial^c \phi(x_0)$,
\begin{align}\label{E:y j not in cut}
\bar y^j &\ne -x_0^j  \qquad \hbox{(or equivalently $\bar y^j
\notin \cCutb (x_0^j)=\Cut(x_0^j)$)}   \qquad \hbox{ for $j = b_0+1,\ldots, k$}.
\end{align}
(If $b_0=a_0$, $\{ a_0 +1, \ldots, b_0\}=\emptyset$.)
\end{enumerate}
After this rearrangement, define
\begin{align*}M'=\bar M' = M^1 \times \ldots \times M^{b_0} \qquad
M''=\bar M''= M^{b_0+1} \times \ldots \times M^k
\end{align*}
The expressions $A'$, $A''$ will be used to denote things defined for elements in $M'$, $M''$, respectively. For example,
\begin{align*}
& y = (y',y'' ), \ \
\bar y = (\bar y', \bar y '') \in M= M' \times M'' = \bar M' \times \bar M'',\\
&  c(y, \bar y) = c'(y', \bar y') + c''(y'', \bar y''),
\end{align*}
and we have the identification
\begin{align*}
T_{\bar x_0}^* \bar M = T^*_{\bar x'_0 } \bar M' \times T^*_{\bar x''_0 } \bar M'', \qquad
\bar M^*(\bar x_0) = \bar M'^*(\bar x_0')\times \bar M''^*(\bar x_0'').
\end{align*}
In the following $n'=\dim M'$, $n''=\dim M''$ and $\pi'$, $\pi''$ denote the canonical
projections from $M$ to $M'$, $M''$, respectively.

\subsection{Analysis near the cut-exposed point.}\label{S:analysis
near cut-exposed}

In this subsection we construct a family of $c$-sections
$Z_{\e,\d,h}$ of $\phi$   near the cut-exposed point $x_0$ defined
in \eqref{x 0}. Regarding these $c$-sections, two important
results (Proposition~\ref{P:d i h i} and \ref{P:regular in ''})
are obtained. In later subsections we will show an Alexandrov type
inequality for $Z_{\e, \d, h}$ which will be paired with the other
Alexandrov type inequality \eqref{eq:left Alex} to lead a
contradiction to the existence of such $x_0$,  thus finishing the proof of Theorem~\ref{T:stay-away}.

Recall the affine function $L$  on $T^*_{\bar x_0^{\cdot \cdot}}
\bar M^{\cdot \cdot}$ given in \eqref{eq:L}. After modifying $L$ by
multiplying it by an appropriate positive constant, there exists a geodesic curve
$[0,1]\ni\d \mapsto \bar x_\d ^{\cdot \cdot} \in \bar M^{\cdot
\cdot} (x_0^{\cdot \cdot})$ starting from $\bar x_0^{\cdot \cdot}$
such that  for the linear map $\nabla L$ on $T^*_{\bar x^\cdd}\bar M^\cdd$,
\begin{align}\label{x d}
\nabla L  (q^\cdd-q_0^\cdd) = \<\frac{\partial}{\partial t}
\Big{|}_{t=0}\bar x_{t}^\cdd , q^\cdd-q_0^\cdd\>.
\end{align}
Consider a $c^\cdot$-segment $[0,1]\ni\e \to \bar x_\e^\cdot \in
\bar M^\cdot $ with respect to $x_0^\cdot$ connecting the point
$\bar x_0^\cdot  $ to  its antipodal point $\bar x_{1/2}^\cdot =
x_0^\cdot$ then to $\bar x_1^\cdot =\bar x_0^\cdot $. ($\bar x_\e$
is nothing else than a closed geodesic starting from $\bar x_0^\cdot$ and
passing through $x_0^\cdot=-\bar x_0^\cdot$ at
$\e=1/2$.) Define
\begin{align}\label{E: bar x e d}
\bar x_{\e, \d}:= (\bar x_\e^\cdot, \bar x_\d^{\cdot \cdot}) \in
M= M^{\cdot} \times M^{\cdot \cdot}.
\end{align}
Obviously $\bar x_{0,0} =\bar x_0$. Two important properties
follow:
\begin{enumerate}
\item[(a)] Since $\bar x_\e^\cdot \in \bar M^\cdot(x_0^\cdot)$
for $\e \in (0,1)$ and  $\bar x_0^\cdd \in \bar
M^\cdd(x_0^\cdd)$ we have
$$
x_0 \in M(\bar x_{\e, \d})= M^\cdot ( \bar x_\e^\cdot) \times  M^{\cdot \cdot} (\bar x_\d^{\cdot \cdot}) \qquad  \hbox{$\forall\, 0 < \e < 1$, $\delta \ge 0$ small}.
$$
\item[(b)] Since $\bar x_{\e, \d}^{\cdot}= \bar x_\e ^\cdot \ne
\bar x_0^\cdot$ for $\e \in (0,1)$, $\bar x_1^\cdot=\bar x_0^\cdot$, and $\bar
x_0^\cdot = -x_0^\cdot$, for every $\e \in (0,1)$ and $\d \in [0,1]$ we have
from Assumption~\ref{A:DASM+} for each $c^i$,
\begin{align}\label{ineq: dasm dot}
-c^\cdot (x^\cdot, \bar x_{\e, \d}^\cdot) + c^\cdot (x_0^\cdot, \bar x_{\e, \d}^\cdot)
\leq - c^\cdot
(x^\cdot, \bar x_0^\cdot) + c^\cdot (x_0^\cdot, \bar x_{0}^\cdot)
\qquad \forall\, x^\cdot \in M^\cdot,
\end{align}
with equality 
only when $x^\cdot =x_0^\cdot$. (See for instance Lemma~\ref{L:whole image} below.)
\end{enumerate}

Consider now the $c$-section $Z_{\e,\d, h}$ obtained by cutting the graph of $\phi$
by the graph of $-c( \cdot , \bar x_{\e, \d})+ c(x_0, \bar x_{\e,\d} ) + h$, that is
\begin{align}\label{eq:Z e d h}
Z_{\e,\d, h} :=\{ x \in M \ | \ \phi  (x) -\phi (x_0) + c(x, \bar
x_{\e, \d}) - c(x_0, \bar x_{\e, \d}) \le h  \}.
\end{align}
As it can be easily seen by moving down the graph of $-c( \cdot , \bar x_{\e, \d})$
and lift it up until it touches the graph of $\phi$, $\bar x_{\e, \d} \in \partial \phi (Z_{\e,\d,h})$.
Hence, thanks to Loeper's maximum principle ({\bf DASM}) we have
\begin{align}\label{E:x e d in c-sub}
\bar x_{\e, \d} \in \partial^c \phi (Z_{\e,\d,h}).
\end{align}
\begin{proposition}
The following equality holds.
 \begin{align}\label{Z e 0 0}
Z_{\e, 0, 0}= S(\bar x_{\e, 0}) = S(\bar x_0) \cap \big{(} \{ x_0^\cd \}\times M^\cdd \big{)}.
\end{align}
 \end{proposition}
\begin{proof}
From \eqref{ineq: dasm dot},
\begin{align*}
 & \phi(x)- \phi (x_0) + c(x, \bar x_{\e, 0})-c(x_0, \bar x_{\e,0})\\
& = \phi(x)- \phi (x_0) + c^\cdot(x^\cdot, \bar x_{\e, 0}^\cdot )-c^\cdot (x_0^\cdot , \bar x_{\e,0}^\cdot) + c^{\cdot \cdot}(x^{\cdot \cdot}, \bar x_{0}^{\cdot \cdot})-c^{\cdot \cdot}(x_0^{\cdot \cdot}, \bar x_{0}^{\cdot \cdot})\\
& \ge \phi(x)- \phi (x_0) + c^\cdot(x^\cdot, \bar x_{ 0}^\cdot )-c^\cdot (x_0^\cdot , \bar x_{0}^\cdot) + c^{\cdot \cdot}(x^{\cdot \cdot}, \bar x_{0}^{\cdot \cdot})-c^{\cdot \cdot}(x_0^{\cdot \cdot}, \bar x_{0}^{\cdot \cdot})\\
&= \phi(x)- \phi (x_0) + c(x, \bar x_{0} )-c (x_0 , \bar x_{0})\ge
0.
\end{align*}
This, together with the equality case for  \eqref{ineq: dasm dot},
yields  \eqref{Z e 0 0}.
\end{proof}

The following  two propositions are essential in our proof of
Theorem~\ref{T:stay-away}.
\begin{proposition}\label{P:d i h i}
Fix $0<  \e < 1$. Then, for any sequences $ \d_i , h_i \to 0$ with $\frac{h_i}{\d_i} \to 0$, we have
\begin{align*}
Z_{\e,\d_i,h_i} \to \{ x_0 \} \hbox{ as $i \to \infty $.}
\end{align*}
\end{proposition}

\begin{proof}
Fix arbitrary sequences $\d_i, \frac{h_i}{\d_i} \to 0$, and denote
\begin{align*}
Z_\infty= \lim_{i\to \infty} Z_{\e, \d_i, h_i}
= \{ z_\infty \in M \ | \ \hbox{there exists a sequence $z_i \in Z_{\e, \d_i, h_i}$ with $z_i \to z_\infty \in M$} \}.
\end{align*}
 By continuity,
$z_\infty \in Z_{\e, 0,0} $ 
for each $z_\infty \in Z_\infty$, and thus by \eqref{Z e 0 0} $z_\infty ^\cdot =x_0^\cdot$.

To show $z_\infty ^\cdd = x_0^\cdd$, we first let $\d > 0$ be sufficiently small and fix a small (closed) neighborhood, say $U$,  of $x_0$ so that all the derivatives (up to the second order) of the function 
 $ U\times [0,1] \ni (x, t) \to c(x, \bar x_{\e, t\d})$ are uniformly bounded.
Then, for  $x\in Z_{\e, \d, h} \cap U$ the following inequalities hold:
\begin{align*}
 h & \ge \phi(x)-\phi(x_0) + c^\cdot (x^\cdot, \bar x_\e^\cdot) - c^\cdot (x_0^\cdot, \bar x_\e ^\cdot) + c^\cdd (x^\cdd, \bar x_\d^\cdd) - c^\cdd (x_0 ^\cdd, \bar x_\d^\cdd)\\
& \ge -c^\cd (x^\cd, \bar x_0^\cd) + c^\cd (x_0^\cd, \bar x_0^\cd) - c^\cdd(x^\cdd, \bar x_0^\cdd) + c^\cdd(x_0^\cdd, \bar x_0^\cdd)\\
& \ \ \  + c^\cdot (x^\cdot, \bar x_\e^\cdot) - c^\cdot (x_0^\cdot, \bar x_\e ^\cdot) + c^\cdd (x^\cdd, \bar x_\d^\cdd) - c^\cdd (x_0 ^\cdd, \bar x_\d^\cdd) \qquad \hbox{ (since $\bar x_0 \in \partial^c\phi(x_0) $)} \\
 & \ge -c^\cdd (x^\cdd, \bar x_0^\cdd) + c^\cdd (x_0^\cdd, \bar x_0^\cdd)
+ c^\cdd (x^\cdd, \bar x_\d^\cdd) - c^\cdd (x_0^\cdd, \bar x_\d^\cdd) \qquad \hbox{(by \eqref{ineq: dasm dot})} \\
& \ge \< D_{\bar x} c^\cdd (x^\cdd, \bar x_0^\cdd) - D_{\bar x} c^\cdd (x_0^\cdd, \bar x_0^\cdd),\frac{\partial}{\partial t}\Big{|}_{t=0}\bar x_{t\d}\> + O(\d^2)\\
&=   \d\,\nabla L ( D_{\bar x} c^\cdd (x^\cdd, \bar x_0^\cdd) - D_{\bar
x} c^\cdd (x_0^\cdd, \bar x_0^\cdd)) + O(\d^2) \qquad \hbox{(by
\eqref{x d} )}
\end{align*}
Use the coordinate $q^\cdd (x^\cdd) = -D_{\bar x^\cdd} c^\cdd (x^\cdd, \bar x_0^\cdd)$ to rewrite this as
\begin{align*}
 \nabla L (q^\cdd - q_0^\cdd) \ge -\frac{h}{\d} -O(\d).
\end{align*}
Since $L(q_0^\cdd)=0$ this gives
\begin{align*}
L (q^\cdd -q_0^\cdd) &\ge -\frac{h}{\d} - O(\d).
\end{align*}
Consider now the sequences $\d_i, \frac{h_i}{\d_i} \to 0$, and any convergent
subsequence of $z_i \in Z_{\e, \d_i, h_i}  \cap U$. For the limit
$z_\infty$, let $q^\cdd = -D_{\bar x^\cdd} c(z_\infty^\cdd, \bar
x_0^\cdd)$. Then $q_\infty^\cdd \in \tilde S^\cdd$  (since $z_\infty \in Z_{\e, 0,0} \subset S(\bar x_0)$ by \eqref{Z e 0 0}), and  from the
above inequality we get
\begin{align*}
 L(q_\infty^\cdd - q_0^\cdd) \ge 0
\end{align*}
which forces $q_\infty^\cdd = q_0^\cdd$ by \eqref{eq:L}. This
shows $z_\infty^\cdd =x_0^\cdd$, and thus $Z_\infty \cap U =\{
x_0\}$. To finish the proof notice that each $Z_{\e,\d,h}$ is path
connected and so is the limit $Z_\infty$. (This path-connectivity
can be seen by noticing that the set $Z_{\e, \d, h}$ is convex in
the coordinates $q(x)= -D_{\bar x} c(x, \bar x_{\e, \d}) \in
T_{\bar x_{\e,\d}}^*\bar M$.) Therefore $Z_\infty =\{ x_0\}$, as
desired.
\end{proof}

\begin{proposition}\label{P:regular in ''}
There exists $\d_0=\d_0(\e) >0$ such
that, if $0\le\d \le \d_0$, $0\le h \le \d^2$, then
for each  $\bar y= (\bar y', \bar y'')\in 
\partial^c\phi(Z_{\e, \d,h}) $ the component $\bar y''$ stays away from the cut-locus of the component $z''$ of $z$ (i.e., $\bar y'' \in \bar M''(z'')$) for every $z \in Z_{\e, \d, h}$. Equivalently $\bar \pi''(\partial^c \phi(Z_{\e,\d,h}) ) \subset \bigcap_{z \in Z_{\e,\d,h}} \bar M''(z'')$.
\end{proposition}
\begin{proof}
Suppose the statement is false along some sequence $\d_i, h_i \to 0 $ with $h_i\leq \d_i^2$,
and let
$x_{i}, z_i \in Z_{\e, \d,h}$, $\bar y_{i} \in \partial^c \phi (x_i)$ be such that $ \bar y_i '' \in \Cut (z_i'')$. Since $Z_{\e, \d_i, h_i}  \to \{ x_0\}$, both $x_i, z_i  \to x_0$.
Moreover if $\bar y_\infty$ is a cluster point for $\{\bar y_i\}_{i\in\N}$,
then $\bar y_\infty \in \partial^c \phi (x_0)$ and $\bar y_\infty'' \in \Cut (x_0'')$. This contradicts
the choice of $M''$ (see \eqref{E:y j not in cut}) and concludes the proof.
\end{proof}

\subsection{An Alexandrov type estimate near the  cut-exposed
point}\label{S:right Alex}

We state the main theorem for the rest of the paper.

\begin{theorem}[\bf Alexandrov lower bound near cut-exposed point]\label{T:right Alex} Fix $0< \e <1$,
and let $Z_{\e,\d,h}$ be as in \eqref{eq:Z e d
h}. There exists $\d_1=\d_1(\e) >0$ so
that, if $0< \d\le \d_1$, then there exists $h_1=h_1(\e,\d)$ such that
\begin{align}\label{eq:right Alex}
h^{\dim M} \lesssim \e^{a_0}
|Z_{\e,\d,h}| |\partial^c \phi (Z_{\e,\d,h})| \qquad \forall\,0< h \le h_1(\e,\d),
\end{align}
where $\lesssim$ is independent of $\e,\d$ and $h$.
\end{theorem}
This result concludes the proof of Theorem \ref{T:stay-away},
since for $\e>0$ small enough and $\d,\frac{h}{\d}\to 0$ we have
$Z_{\e,\d,h} \to \{x_0\}$ (by Proposition~\ref{P:d i h i}),
and \eqref{eq:right Alex}
is in contradiction with \eqref{eq:left Alex}.

The following subsections are devoted to the proof of Theorem
\ref{T:right Alex}, that we divide into three parts. First, in
Section~\ref{S:cut-locus component} we get Alexandrov type
estimates for the sets obtained by the intersection of
$Z_{\e,\d,h}$ with the cut-locus components of $x_0$. In
Section~\ref{S:regular component}, we analyze the projection
$\pi''(Z_{\e,\d,h})$ of $Z_{\e,\d,h}$ onto the regular component
$M''$ of $x_0$. We construct a suitable convex set, say $\tilde C$,
which has size comparable to the image $\partial^c \phi (Z_{\e,\d,h})$,
and we get a version of the estimate \eqref{eq:right Alex} involving
$\tilde C$ and $\pi''(Z_{\e,\d,h})$ (see Proposition~\ref{P:tilde
C}(3)). Finally in
Section~\ref{S:final argument} we combine these results and conclude the proof.

\subsection{Proof of Theorem~\ref{T:right Alex} (Alexandrov lower bound near cut-exposed point): analysis in
 the cut-locus component $M'$}\label{S:cut-locus component} The main result
of this section is Proposition~\ref{P:|Z h i|} that gives an
Alexandrov type estimate for the intersection of $Z_{\e,\d,h}$
with the cut-locus components of $x_0$.

We start with a few elementary results.

\begin{lemma}\label{L:whole image}
Let $\S^n$ be the standard round sphere, and $c (x, \bar x) =
f(\dist (x, \bar x))$ for $(x, \bar x) \in \S^n \times \S^n$, where
$f$ is a  smooth strictly increasing function $f:\R_+ \to \R_+$. Assume that $c$ satisfies
 Assumption~\ref{A:DASM+} ({\bf DASM$^+$}). Then, for every $x, \bar x \in \S^n$,
\begin{align*}
-c(-\bar x, \bar y) + c(-\bar x, \bar x) \ge - c (x, \bar y) +
c(x, \bar x)  \qquad \forall\,\bar y \in \S^n,
\end{align*}
where $-\bar x$ denotes the antipodal point of $\bar x$.
Moreover equality holds if and only if $\bar y=\bar x$.
\end{lemma}
\begin{proof}
For any $x, \bar x \in \S^n$, one can find a $c$-segment $x(s)$
with respect to $\bar x$ such that $x(0)=x(1)= -\bar x$ and
$x(s_0) = x$ for  some $s_0 \in [0,1]$. The inequality (together with the characterization of
the equality case) then
follows from ({\bf DASM$^+$}) for the function $\bar m_s (\cdot) = -
c(x(s), \cdot) + c(x(s), \bar x)$.
\end{proof}

For each $1\le i \le k$ and $z \in M$, let $M^i_{z}$ denote the $i$-th slice of $M$ through $z$, that is
\begin{align*}
 M^i_{z} := \{ x \in M \ | \ x^j =z^j \hbox{ for $j\ne i$}  \}.
\end{align*}
The following lemma generalizes the fact that on $M=\bar M=\S^n$ with $c=\dist^2/2$,
if $x \in \S^n$ and $-x \in \p^c\phi(x)$,
then $\p^c\phi(x)=\S^n$.

\begin{lemma}\label{L:phi on M i x 0}
Let $M,\bar M,c$ be as in Theorem~\ref{T:stay-away}.  Let $\phi$ be a
$c$-convex function on $M$. Fix $z = (z^1, \ldots, z^k) \in M=M^1
\times \ldots \times M^k$ and an open set $U$ with $z \in U$. Fix
$i \in \{1, \ldots, k\}$, and let $\bar z \in M$ with $\bar z^i
=-z^i$. The following holds:
\begin{itemize}
\item[(1)]  If $\bar z \in [\partial^c \phi(U)]_z$ (resp. $\bar z
\in \partial^c \phi (z)$),  then $M^i_{\bar z} \subset [\partial^c
\phi(U)]_z $ (resp. $M_{\bar z}^i \subset \partial^c \phi(z)$).
\item[(2)] Suppose $\bar z \in \partial^c \phi (z)$. Then, for
each $x \in M^i_{z}$,
$\phi(x) -\phi (z) = -c^i(x^i, - z^i) + c^i( z^i, - z^i)$. 
\end{itemize}
\end{lemma}
\begin{proof}
To prove (1) it is enough to observe that for $\bar x\in M^i_{\bar z}$ and $x \in M$,
\begin{align*}
&- c(x, \bar x) + c(z,  \bar x)\\
&= -c^i(x^i, \bar x^i) + c(z^i, \bar x^i) + \sum_{j\ne i}\big{[} -c^j(x^j, \bar z^j) + c^j (x^j, \bar z^j)\big{]}  \\
&\le -c^i(x^i, - z^i) + c(z^i, - z^i) + \sum_{j\ne i}\big{[} -c^j(x^j, \bar z^j) + c^j (z^j, \bar z^j)\big{]} \qquad \hbox{(by Lemma~\ref{L:whole image})}
\\&= - c(x, \bar z) + c(z,  \bar z) \qquad \hbox{(since $\bar z^i = -z^i$)}.
\end{align*}
 The last line is bounded from above by $\phi(x) - \phi(z)$ either if $x \in\partial U$ or $\bar z \in \p^c \phi (z)$.

Let us prove the (2). Suppose $\bar z \in \partial^c\phi (z)$.
 By duality  (Lemma~\ref{L:reci}), $z \in \partial^{\bar c} \bar \phi (\bar z)$ for the dual $\bar c$-convex function $\bar \phi$. Applying (1) to $\bar \phi$ we get
$M^i_z \in \partial^{\bar c}\bar \phi (\bar z)$, or equivalently
$M^i_z \subset S(\bar z)$. Therefore for all $x \in M^i_{z}$ we have
 \begin{align*}
  \phi(x) -\phi (z) &= -c(x, \bar z) + c( z, \bar z) \\&=
   -c^i(x^i, - z^i) + c^i( z^i, - z^i) \qquad \hbox{ (since $x^j=z^j$ for $j\ne i$)}
\end{align*}
which concludes the proof.
\end{proof}

Let $i \in \{ 1, \ldots,  b_0\} $, i.e., $M^i$ is a component of
$M'$. Recall that $x_0$ is the cut-exposed point defined in \eqref{x 0}. By definition of
$b_0$ in \eqref{E:y j in cut} and \eqref{E:y j not in cut}, there
exists $\bar y_i \in \partial^c\phi (x_0)$ such that $\bar
y_i^i=-x^i_0$. (If $i\le a_0$ then one can choose $\bar y_i = \bar
x_0$.) Let   $Z_{\e,\d,h}^i:=\pi^i \bigl(Z_{\e,\d,h} \cap M^i_{x_0}\bigr)$ for the canonical projection $\pi^i : M \to M^i$. Then
Lemma~\ref{L:phi on M i x 0}(2) implies
\begin{align}\label{E:Z e d h i}
Z^i_{\e,\d,h} = \{ x^i \in M^i \ | \
 -c^i(x^i, -x_0^i) + c^i(x^i_0, - x_0^i) + c^i(x^i, \bar x^i_{\e,\d}) - c^i(x_0^i, \bar x_{\e, \d}^i) \le h \}.
\end{align}
Here comes the main result of this section.
\begin{proposition}\label{P:|Z h i|}
There exist $\d_2=\d_2(\e) >0$ such
that, if $0< \d\le \d_2$, then there exists $h_2=h_2(\e,\d)$
such that
the set $Z_{\e,\d,h}^i$ satisfies the following estimates for $0<h\leq h_2$:
\begin{align*}
h^{\dim M^i} & \lesssim \e |Z_{\e,\d, h}^i | \qquad \hbox{if $1\le i \le a_0$,}\\
h^{\dim M^i} & \lesssim  |Z_{\e,\d, h}^i | \qquad \hbox{if $a_0 +1
\le i \le b_0$,}
\end{align*}
where $\lesssim$ is independent of $\e, \d$ and $h$ and $|Z_{\e,\d,h}^i|$ denotes the Riemannian volume in the submanifold  $M^i$.
\end{proposition}
\begin{proof}
From \eqref{E:Z e d h i} and Lemma~\ref{L:whole image}
we have $Z_{\e,\d,h}^i \to \{x_0^i\}$ as $h\to 0$. Thus for sufficiently small $h$ we can
embed $Z^i_{\e,\d,h}$ into $\in T^*_{ x_0 ^i } M^i$ by $x^i
\mapsto q^i(x^i) = -D_{\bar x^i} c^i(x^i,  x_0^i)$. Let $W_h^i$ be
its image. Then
\begin{align*}
|W_h^i | \le \Bigl(\max_{\ x^i \in Z_{\e,\d,h}^i} \bigl|-D_{x^i}D_{\bar x^i}   c^i (x^i, x_0^i ) \bigr|\Bigr)\, |Z_{\e,\d,h}^i|\lesssim |Z_{\e,\d,h}^i|
\end{align*}
for $h$ sufficiently small.
In the following we bound $|W_h^i|$ from below.

 Without loss of generality, assume $M^i$ is the unit sphere.
 Let $q_0^i = q^i(x_0^i)$. By abuse of notation
use $c^i (q^i, \bar x^i)$ to denote $c^i(x^i(q^i), \bar x^i)$, and renormalize this cost function as
$$
c_h^i (q^i, \bar x^i) = \frac{1}{h} \big{[}c^i( h q^i + q_0^i, \bar x^i) -c^i(q_0^i, \bar x^i)\big{]}
$$
Then \eqref{E:Z e d h i} implies
 $W_h^i = h \hat W_h^{i} + q_0^i$, where
$$
\hat W_h^i := \{ q^i \in T^*_{x_0^i }  M^i \ | \  -c^i_h(q^i,-
x^i_0) + c_h^i(q^i, \bar x^i_{\e,\d})  \le 1 \}
$$
Recall $c^i=f^i (\dist_i)$ for some smooth nonnegative uniformly
convex function $f^i :\R_+\to \R_+$ such that $f^i(0)=0$, $\frac{df^i}{dt}(0)=0$.
Thus, as  $h \to 0$ the
renormalized cost $-c^i_h ( q^i, - x_0^i)$ converges to the
conical function
$$
q^i\mapsto \frac{d f^i}{dt}(\pi)|q^i|, \hbox{ for $q^i \in T^*_{ x^i_0}
M^i$}.
$$
(Here, we used $\dist_i(x_0^i, -x_0^i) =\pi$.)

{\em Case I:} If $1\leq i\le a_0$, then $\bar x^i_{\e,\d}=\bar x^{\cd\, i}_\e$, and so
$c^i_h (q, \bar x^i_{\e,\d})$ converges to the linear function
$$
q^i \mapsto D_qc^i(q_0^i, \bar x^{\cd \, i}_\e) \cdot q^i
$$
where
$$
|D_q c^i(q_0^i, \bar x^{\cd \, i}_{\e})| =\frac{d
f^i}{dt}(\pi- 2\pi\e)  \ge \frac{df^i}{dt}(\pi) -C\e
$$
for some constant $C>0$.
(Here, we used $\dist(x_0^i,  \bar x^{\cd \, i}_\e) =\pi- 2\pi\e$.)
Therefore in the limit $h \to 0$ one can easily check that
\begin{figure}[h]
\centerline{\epsfysize=1.5truein \epsfbox{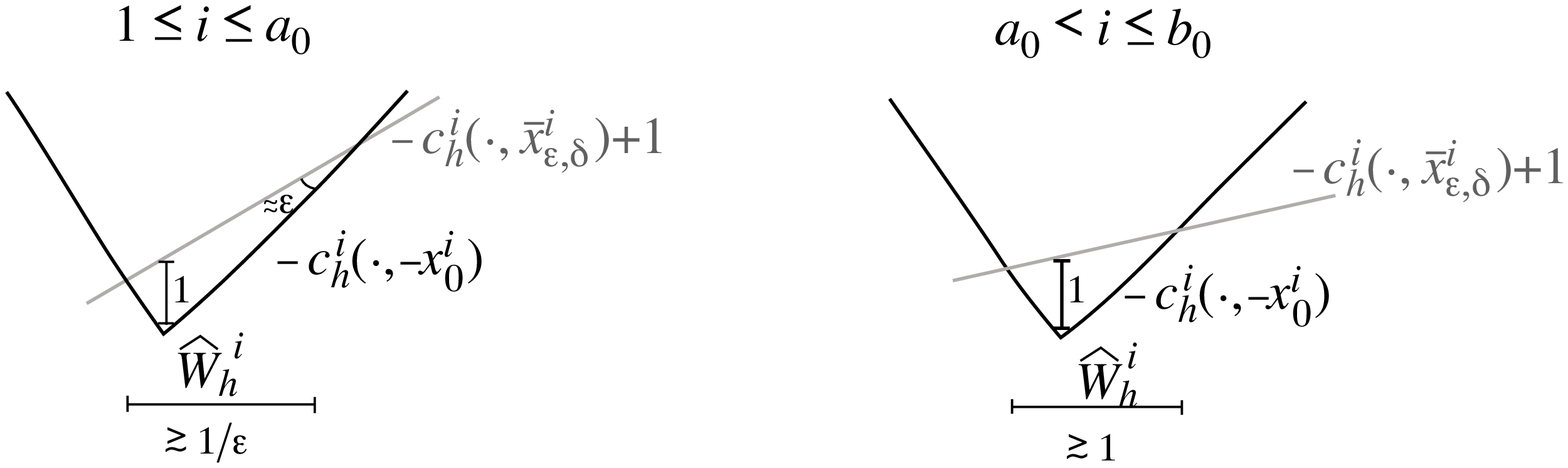}}
\caption{{\small If $1 \leq i \leq a_0$ then $-x_0^i=\bar x_0^i$ and $\dist_i(-x_0^i,\bar x^i_{\e,\d})\approx \e$,
so the size of the section is of order $1/\e$ (see also Figure 1). On the other hand, if $a_0<i\leq b_0$ then $-x_0^i \neq \bar x_0^i$,
which implies that $\dist_i(-x_0^i,\bar x^i_{\e,\d})$ is uniformly bounded away from $0$, and the size of the section is of order $1$.
}}
\end{figure}
$$
\lim_{h\to 0}|\hat W_h^i| \gtrsim  \frac{1}{\e},
$$
and thus for $h > 0$ sufficiently small
$$
|W_h^i| = h^{\dim M^i} |\hat W_h^i| \gtrsim \frac{h^{\dim
M^i}}{\e}.
$$

{\em Case II:} If $a_0 < i \le b_0$, then $\bar x_{\e,\d}^i=\bar x^{\cdd\,i}_\d$. Similarly as for the above case,   $c^i_h (q, \bar x^i_{\e,\d})$ converges to the linear function
$$
q^i \mapsto D_qc^i(q_0^i, \bar x^{\cdd \, i}_\d) \cdot q^i.
$$
Since $\bar x_\d^\cdd \in \bar M^\cdd (x_0^\cdd)$ for $\d >0$ small enough,
there exist positive constants $C,C_0$ such that
\begin{align*}
&|D_q c^i(q_0^i, \bar x^{\cdd \, i}_{\d})| \le
\frac{df^i}{dt}(\pi-C)  \leq \frac{df^i}{dt}(\pi)-C_0,
\end{align*}
where for the last inequality we used the uniform convexity of
$f$. From this one can check that $ \lim_{h\to 0}|\hat W_h^i|
\gtrsim 1$, and thus for sufficiently small $h>0$
$$\qquad |W_h^i| \gtrsim  h^{\dim M^i}.$$ This concludes the proof.
\end{proof}

\subsection{Proof of Theorem~\ref{T:right Alex}  (Alexandrov lower bound near cut-exposed point): analysis in
the regular component $M''$}\label{S:regular component} The main result of
this subsection is Proposition~\ref{P:tilde C}. Fix $0< \e <1$, and assume that
$\d$ and $\frac{h}{\d}$ are sufficiently small so that, as in
Proposition~\ref{P:regular in ''}, the set $Z_{\e,\d,h}$ is close
to the cut-exposed point $x_0$, and so in particular $Z_{\e,\d,h} \subset
M(\bar x_{\e,\d})$. Consider the change of coordinates $q \in T^*_{\bar
x_{\e,\d}} \bar M \mapsto x(q)  \in M(\bar x_{\e,\d})$ induced by the relation
\begin{align}\label{E:q e}
q = - D_{\bar x} c(x(q), \bar x_{\e,\d}),
\end{align}
and let $\tilde Z_{\e,\d,h} \in T^*_{\bar x_{\e,\d}} \bar M$ be the set $Z_{\e,\d,h}$ in this chart.
The function $\phi$ and the cost $c$ are transformed to
\begin{align*}
\varphi (q) = \phi (x(q) ) + c (x(q), \bar x_{\e,\d}),
\end{align*}
and
$$
\tilde c(q, \bar y) = c(x(q), \bar y)  - c(x(q), \bar x_{\e,\d})
\qquad \hbox{for $(q, \bar y) \in T^*_{\bar x_{\e,\d}} \bar M \times
\bar M$}.
$$
Notice that
$$
\tilde c(q, \bar x_{\e,\d}) \equiv 0,
$$
and $\varphi$ is a $\tilde c$-convex function on $T^*_{\bar x_{\e,\d}} \bar M$. Moreover
$$
\tilde Z_{\e, \d,h}=\{ q \in T^*_{\bar x_{\e,\d}} \bar M \ | \ \varphi(q) -\varphi(q_0) \le h \}
$$
where $q_0$ is the point corresponding to $x_0$ in this new chart.
It is important to recall that, thanks to  Assumption~\ref{A:convex DASM} ({\bf convex DASM}), $\tilde c$
and $\varphi$ are convex. (See Section~\ref{SS:coordinate})

We have the natural decomposition (with obvious notation)
\begin{align}\label{E:q' e}
q = (q' , q'') &= (-D_{\bar x'} c'(x'(q'), \bar x_{\e,\d}'), -D_{\bar x''} c''(x''(q''), \bar x_{\e,\d}''))\\\nonumber
 & \in T^*_{\bar x_{\e,\d}} \bar M = T^*_{\bar x_{\e,\d}'} \bar M' \times T^*_{\bar x_{\e,\d}''} \bar M''.
\end{align}
(Here, one should keep in mind that, by the definition of $\bar x_{\e,\d}$, the component $\bar x_{\e,\d}''$ in $M''$ does not depend on $\e$.) 
The modified cost $\tilde c(q, \bar y)$  has the decomposition
$$
\tilde c (q, \bar y)= \tilde c' (q', \bar y') + \tilde c '' (q'', \bar y'')
$$
where
$$
\tilde c' (q', \bar y') = c'(x'(q'), \bar y') - c'(x'(q'), \bar
x_{\e,\d}')\qquad \hbox{for $q' \in T^*_{\bar x_{\e,\d}'} \bar M'$,}
$$
and
$$
\tilde c''(q'', \bar y'') = c'(x''(q''), \bar y'') - c''(x''(q''),
\bar x_{\e,\d}'') \qquad\hbox{for $q'' \in T^*_{\bar x_{\e,\d}''} \bar M''$.}
$$
Let $\tilde \pi', \tilde \pi''$ denote the canonical projection
from $T^*_{\bar x_{\e,\d}} \bar M$ onto $T^*_{\bar x_{\e,\d}'} \bar M'$ and $T^*_{\bar x_{\e,\d}''} \bar M''$, respectively.

Now, let us construct a convex set $\tilde C \subset
T^*_{q_0}(T^*_{\bar x_{\e,\d}} \bar M)$ that we will use later to
estimate $|\partial^c \phi(Z_{\e,\d,h})|$ from below (see
Proposition~\ref{P:bound partial c by tilde C}).
The strategy of the proof follows the lines of the one of \cite[Proposition 7.6]{FKM}.

\begin{proposition}\label{P:tilde C}
Fix $0< \e <1$, and assume that $0<\d\leq \d_0$ and $0<h\leq \d^2$,
with $\d_0$ as in Proposition~\ref{P:regular in ''}. Then there exists a
convex set $\tilde C \in T^*_{q_0}(T^*_{\bar x_{\e,\d}} \bar M)$ satisfying
the following properties:
\begin{itemize}
\item[(1)] $\tilde C \subset \{0 \} \times T^*_{q_0''} (T^*_{\bar
x''_{\e,\d}} \bar M'') \subset T^*_{q_0'}(T^*_{\bar x_{\e,\d}'} \bar M')
\times T^*_{q_0''}(T^*_{\bar x_{\e,\d}''} \bar M'');$ 
\item[(2)] $ 
  \tcexp_{q_0} \tilde C = \{\bar z \in M | -\partial_q \tilde c (q_0, \bar z) \cap \tilde C \ne \emptyset \} \subset [\partial^{c} \phi (Z_{\e,\d,h})]_{x_0} \subset \partial^{c} \phi (Z_{\e,\d,h}),$  where $\partial_q$ denotes the subdifferential with respect to $q$ variable;
\item[(3)] $\Haus{n''} (\tilde C) \Haus{n''}(\tilde \pi''(\tilde
Z_{\e,\d,h})) \gtrsim h^{n''},$ where $\gtrsim$ is independent of
$h, \d, \e$.
\end{itemize}

\end{proposition}
\begin{proof}
In the following, we first construct such a set $\tilde C$ and
then we show the desired properties. The set $\tilde C$ will be
given as a convex hull of certain covectors $\hat p_1, \ldots,
\hat p_{n''}$, see \eqref{eq:hat p i at q 0}. We go through
several steps.

First we find some auxiliary covectors $p_1, \ldots, p_{n''}$.
From Lemma~\ref{L:John} applied to the convex set
$\tilde \pi''(\tilde Z_{\e,\d,h})$, there is an ellipsoid $\tilde
E$ such that
\begin{align}\label{E:ellipse}
\tilde E \subset \tilde \pi''( \tilde Z_{\e,\d,h}) \subset n''  \tilde E
\end{align}
 where the scaling $n'' \tilde E$ is  with respect to the barycenter of the ellipsoid.
Let $p_i ''$, $ 1 \le i \le n''$, denote the unit orthogonal
covectors parallel to the axes of the ellipsoid $\tilde E$, and denote by
$a_i$ the length of the $i$-th principal axis of $\tilde
E$. Find hyperplanes $\Pi_i '' \subset  T^*_{\bar x''_{\e,\d} } \bar M''$ that
are orthogonal to $p_i''$ and touch tangentially the boundary
of $\tilde \pi''(\tilde Z_{\e,\d,h})$ at points $q_i''$,
$1 \le i \le n''$. Let $q''_0$ be the point in $T^*_{\bar
x''_{\e,\d}} \bar M''$ corresponding to $x_0$, and denote by $\ell_i$ the
distance from $q''_0$ to $\Pi_i ''$.
Then, thanks to \eqref{E:ellipse} we have
\begin{align}\label{eq:Haus tile Z h''}
\prod_i^{n''} \ell_i \le \prod_i^{n''} (2n'' a_i) \lesssim \Haus{n''}
(\tilde \pi''(\tilde Z_{\e,\d,h})).
\end{align}
For each $q_i''$, there exists $q_i' \in T^*_{\bar x_{\e,\d}'} \bar M'$
such that the hyperplane $\Pi_i:=T^*_{\bar x_{\e,\d}'} \bar M'\times \Pi_i''\subset
T^*_{\bar x_{\e,\d}} \bar M$ tangentially touches the boundary  $\partial \tilde
Z_{\e,\d,h}$ at the point $q_i = (q_i', q_i'')$.  Let $x_i = \cexpb_{\bar x_{\e,\d}} q_i$.
Since $p_i =(0,p_i'')$ is orthogonal to $\Pi_i$ and
$\tilde Z_{\e,\d,h}$ is a sublevel set of the convex function
$\varphi$, there exists a scalar multiple $t_i \in \R_+$ such that
$t_i p_i \in \partial \varphi (q_i)$.
 By  Assumption~\ref{A:c-exp} and Loeper's maximum principle ({\bf DASM})  (Lemma~\ref{L:DASM for c-sub}),
  the point $\bar z_i = \tcexp_{q_i} t_i p_i$ satisfies $\bar z_i \in \partial^{\tilde c} \varphi (q_i) = \partial^c \phi (x_i)$. Note that in fact,
 \begin{align*}
\bar z_i &= \tcexp_{q_i} t_i p_i = \cexp_{x_i } \eta (t_i p_i)
\end{align*}
  where $\eta$ is the affine map given by Lemma~\ref{L:eta shift} (in whose statement we replace $x_0$, $q_0$ and $\bar y_0$ with $x_i$, $q_i$ and $\bar x_{\e,\d}$, respectively).
Moreover, using the decomposition 
\begin{align*}
p_i = (0, p_i'')  \in T^*_{q_i'}(T^*_{\bar x_{\e,\d}'} \bar M') \times T^*_{q_i''}(T^*_{\bar x_{\e,\d}''} \bar M''),
\end{align*}
we see that the $\tilde c$-segment  (with respect to $q_i$)
$$
[0,1] \ni t \mapsto \bar z_i (t) =  \tcexp_{q_i} (1-t) t_i p_i
= \cexp_{x_i}\big{(}(1-t) \eta (t_i p_i)\big{)}
$$ from $\bar z_i (0) = \bar z_i$ to $\bar z_i (1) = \bar x_{\e,\d}$, is of the form
\begin{align*}
\bar z_i (t) = (\bar x'_{\e,\d}, \bar z_i''(t)) \in \bar M' \times \bar M''.
\end{align*}
Observe that by Proposition~\ref{P:regular in ''} and Assumption~\ref{A:domain convex}, we have 
\begin{align}\label{eq:bar z i '' stay away}
\bar z_i '' (t) \in \bar M'' (x''), \qquad \forall t \in [0,1], \forall x \in Z_{\e, \d, h} .
\end{align}


 We use these $\tilde c$-segments $\bar z_i(t)$ to define the points $\hat p_i$, $i=1, \cdots, n''$. 
Define
the function
$$
m_{\bar z_i (t)} (q):= -\tilde c( q , \bar z_i (t)) + \tilde c(q_i, \bar z_i (t)) + \varphi (q_i) .
$$
Clearly, $m_{\bar z_i(0)} \le \varphi$ and $m_{\bar z_i(1)} \equiv \varphi (q_i) =h+\varphi(q_0)$. By continuity there exists $\t_i \in [0, 1)$  such that
\begin{align*}
m_{\bar z_i (\tau_i)} (q_0 ) =\varphi(q_0).
\end{align*}
\begin{figure}[h]
\centerline{\epsfysize=1.5truein \epsfbox{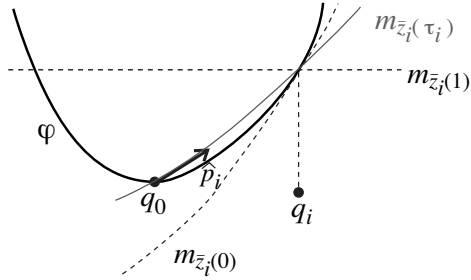}}
\caption{{\small
The supporting function $m_{\bar z_i(0)}=m_{\bar z_i}$ touches $\var$ at $q_i$ from below. By interpolating between $m_{\bar z_i}=m_{\bar z_i(0)}$ and $m_{\bar x_{\e,\d}}=m_{\bar z_i(1)}$
along the $\tilde c$-segment with respect to $q_i$, we can find  $\tau_i \in [0,1)$ such that $m_{\bar z_i(\tau_i)}(q_0)=\var(q_0)$.
Then the covector $\hat p_i$ used to construct $\tilde C$ is defined as 
 $\hat p_i := (0, -D_{q''} \tilde c'' (q_0 '',\bar z_i '' (\tau_i)) \in \partial m_{\bar z_i (\tau_i)} (q_0) $.}}
\end{figure}
Also, Loeper's maximum principle ({\bf DASM}) implies
\begin{align*}
m_{\bar z_i (\t_i)}  & \le \max[h + \varphi(q_0), \varphi ],
\end{align*}
so that in particular
\begin{align*}
m_{\bar z_i (\t_i)}& \le  \varphi 
\hbox{ on $\partial \tilde Z_{\e,\d,h}$,}
\end{align*}
 hence,
by the definition of $[\partial^c \phi (Z_{\e,\d,h})]_{x_0}$,
\begin{align}\label{eq:in c-sub wrt x 0}
\bar z_i (\t_i) \in [\partial^{\tilde c}\varphi (\tilde Z_{\e,\d,h})]_{q_0} = [\partial^c \phi (Z_{\e,\d,h})]_{x_0} \qquad \hbox{for every $i=1, \ldots, n''$}.
\end{align}
For later use, consider the nonzero vectors 
\begin{align}\nonumber
p_i (\t_i) & = (1-\t_i) t_i p_i\\\nonumber
& = (0, (1-\t_i) t_i p_i'')  \\\label{eq:p i tau i}
& = \big{(}0, -D_{q''} \tilde c''( q_i'', \bar z_i '' (\t_i))\big{)} \in T^*_{q_i'}(T^*_{\bar x_{\e,\d}'} \bar M') \times T^*_{q_i''}(T^*_{\bar x_{\e,\d}''} \bar M''), \qquad i = 1, \cdots, n''.
\end{align}
Clearly these vectors are all mutually orthogonal. Moreover, because
$$
p_i (\t_i)  \in \partial m_{\bar z_i (\t_i)} (q_i), \qquad i=1, \cdots, n'',
$$ 
 we have by the convexity of  $m_{\bar z_i (\t_i)}$,
\begin{equation}
\label{E:D i m i ge h / l i}
 |p_i (\t_i) |  \geq \frac{\varphi(q_i)-\varphi(q_0)}{\dist(q_0,\Pi_i)}=\frac{h}{\ell_i}.
\end{equation}
To finish the construction of $\tilde C$, let
 \begin{align}\label{eq:hat p i at q 0}
\hat p_i & := (0,  -D_{q''}\tilde c'' (q_0'',\bar z_i'' (\tau_i)) \\ \nonumber  & \in T^*_{q_0'}(T^*_{\bar x_{\e,\d}'} \bar M') \times T^*_{q_0''}(T^*_{\bar x_{\e,\d}''} \bar M''),  \qquad i=1,
\ldots, n''.
\end{align}
Notice that $\bar z_i (\t_i) =(x_{\e,\d}', \bar z_i(\t_i)'') = \tcexp_{q_0} \hat p_i$.
Let $ \tilde C = \co (\hat p_1 , \ldots, \hat p_{n''})$
 be the convex hull of $\hat p_1 , \ldots, \hat p_{n''}$. In the following, we will see that $\tilde C$ satisfies the desired properties (1), (2) and (3).
First, (1) follows immediately from  \eqref{eq:hat p i at q 0}, while
(2) is a direct consequence of \eqref{eq:in c-sub wrt
x 0} and  Lemma~\ref{L:DASM for c-sub of a set}.

Now, let us show (3).  By \eqref{eq:bar z i '' stay away}
each $\bar z_i(\t_i)''$ stays uniformly away (for small $\d, h$)  from the cut-locus of
 $\pi''(Z_{\e,\d,h})$. 
Hence we can apply Lemma~\ref{L:c estimate} to
 \eqref{eq:p i tau i} and \eqref{eq:hat p i at q 0}  to see that
$\hat p_i$ is close to  $p_i (\t_i)$ when we use the canonical identification $T^*_{q_i} (T^*_{x_{\e,\d}} \bar M) \approx T^*_{q_0} (T^*_{x_{\e,\d}} \bar M)$; more precisely,
$$
 \bigl|\hat p_i - p_i (\t_i)\bigr| \leq o_h(1)|\hat p_i|,
$$
where $o_h(1)$ is a quantity which goes to $0$ as $h \to 0$.
Since the vectors
 $\bigl\{  p_i (\t_i)  \bigr\}_{i=1,\ldots,n''}$ are all mutually orthogonal,
$\hat p_i$ are almost mutually orthogonal covectors, which by \eqref{E:D i m i ge h / l i} satify
$$
|\hat p_i | = |\hat p_i''|  \gtrsim \bigl| p_i (\t_i) \bigr| \ge \frac{h}{\ell_i}.
$$
(Here, for sufficiently small $\d$ and $\frac{h}{\d}$, the inequality
$\gtrsim$ and the almost orthogonality are independent of $\d$,
$h$ and $\e$.) This gives
$$
\Haus{n''} (\tilde C) \gtrsim \prod_{i=1}^{n''} \frac{h}{\ell_i}.
$$
This estimate combined with \eqref{eq:Haus tile Z h''} shows (3).
This completes the proof. 
\end{proof}

\subsection{Proof of Theorem~\ref{T:right Alex} (Alexandrov lower bound near cut-exposed point): final
argument }\label{S:final argument} In this section we finish the
proof of Theorem~\ref{T:right Alex}. Let $0< \e < 1$, and fix
$0<\d\leq \d_1(\e):=\min\{\d_0(\e),\d_2(\e)\}$ and $0<h\leq h_1(\e,\d):=\min\{\d^2,h_2(\e,\d)\}$,
with $\d_0(\e)$ and $\d_2(\e),h_2(\e,\d)$ as in
Proposition~\ref{P:regular in ''} and \ref{P:|Z h i|} respectively. The estimates
$\lesssim$, $\gtrsim$, $\approx$ in this section are all
independent of $\e,\d$ and $h$.

To make use of the results of previous sections, we need the
following comparison result:
\begin{proposition}\label{P:bound partial c by tilde C}
The set $\tilde C$ constructed in Proposition~\ref{P:tilde C} satisfies
\begin{align*}
  \Haus{n''} (\tilde C) \lesssim |\partial^c \phi (Z_{\e,\d,h})|.
\end{align*}
\end{proposition}
Note that even with Proposition~\ref{P:tilde C} (2), this estimate is not obvious because $n'' < \dim M$.

\begin{proof}
For each $\e,\d,h$  as in Proposition~\ref{P:tilde C}, we will find an auxiliary set
$\mathcal{A}=\mathcal{A}_{\e,\d,h} \subset \mathcal{D}$ in a fixed
(thus independent of $\e,\d,h$) compact set $\mathcal{D}\subset
M^*(x_0) \subset T^*_{x_0} M$ such that
\begin{align}\label{E:A in partial c}
\cexp_{x_0} (\mathcal{A}) & \subset [\partial^c \phi
(Z_{\e,\d,h})]_{x_0} \subset \partial^c \phi
(Z_{\e,\d,h});\\\label{E:A bounded by tilde C} |\mathcal{A}| &
\gtrsim \Haus{n''}(\tilde C).
\end{align}
Once such a set is constructed, the desired estimate follows from
\begin{align*}
|\partial^c \phi (Z_{\e,\d,h}) | \ge |\cexp_{x_0} ( \mathcal{A} )|
\gtrsim  | \mathcal{A}| \qquad \hbox{(since $\mathcal{A} \subset
\mathcal{D}$)}.
\end{align*}
The construction of $\mathcal{A}$ goes through several steps.
First, apply to the set $\tilde C$ the (extended) map $p \in
T^*_{q_0}(T^*_{\bar x_{\e,\d}} \bar M) \mapsto  \eta(p) \in T^*_{x_0} M$
as in Lemma~\ref{L:eta shift} (with $\bar y_0 = \bar x_{\e,\d}$),
and let $\eta(\tilde C) \subset T^*_{x_0} M$ denote its image.
Notice that by Proposition~\ref{P:tilde C}(2)
$$
\cexp_{x_0} (\eta (\tilde C))  = \tcexp_{q_0} \tilde C \subset [\partial^c \phi(Z_{\e,\d,h})]_{x_0}.
$$
Let us compare $\Haus{n''} (\eta(\tilde C))$  with $\Haus{n''} (\tilde C)$.
For each $p= (0, p '') \in \tilde C$, Lemma~\ref{L:eta shift} applies as
$$
\eta (p)= \big{(} \eta_{\e, \d}' \ , \ p ''\bigl( -D_{x''} D_{\bar x''}
c'' (x_0'', \bar x''_{\e,\d})\bigr) + \eta_{\e,\d}''\big{)},
$$
where $\eta_{\e,\d}=-D_x c (x_0, \bar x_{\e,\d})$ (thus,
$\cexp_{x_0} (\eta_{\e,\d}) = \bar x_{\e,\d}$). Therefore
\begin{align*}
\eta(\tilde C)
\subset \{ \eta_{\e,\d}'\} \times  T^*_{x_0''}M
\end{align*}
and
\begin{align*}
\Haus{n''} (\eta(\tilde C))=  |\det D_{x''} D_{\bar x''} c'' (x_0'', \bar x''_{\e,\d})| \Haus{n''} (\tilde C)
\approx \Haus{n''} (\tilde C).
\end{align*}
Notice that  $\bar x_{\e,\d}''$ is independent of $\e$ (see \eqref{E: bar x e d})
and stays uniformly away from $\Cut (x_0'')$, so that the above estimate is
independent of $\e, \d$ and $h$.

We now use a convexity argument to construct $\mathcal{A}$. We
will first construct some suitable sets $C^1, \ldots, C^{b_0}$, and
$\tilde C_0$, inside a fixed compact set (independent of
$\e,\d,h$) in $ M^*(x_0)$, which satisfy the properties of the
sets $S_i$ in Lemma~\ref{lemma:mutiple orthogonal sections}. These
sets will also satisfy:
\begin{align*}
& \cexp_{x_0}(C^1) \cup \ldots \cup \cexp_{x_0}( C^{b_0}) \cup \cexp_{x_0}(\tilde C_0 ) \ \  \subset \ \ [\partial^c \phi (Z_{\e,\d,h})]_{x_0};\\
& \Haus{n_i}(C^i) \gtrsim 1, \qquad i=1, \ldots b_0;\\
& \Haus{n''}( \tilde C_0) \gtrsim \Haus{n''}(\eta(\tilde C)).
\end{align*}
Then $\mathcal{A}$ will be given as the convex hull of these sets, that
is $\mathcal{A} = \co(C^1, \ldots, C^{b_0},\tilde C_0 )$. By
convexity of $M^*(x_0)$, $\mathcal{A}$ will be in a fixed compact set,
say $\mathcal{D}$, independent of $\e,\d,h$, and the $c$-convexity
of $[\partial^c \phi (Z_{\e,\d,h})]_{x_0}$  (see Lemma~\ref{L:DASM for c-sub of a set}) will imply
$\cexp_{x_0}(\mathcal{A}) \subset [\partial^c \phi
(Z_{\e,\d,h})]_{x_0}$, showing \eqref{E:A in partial c}. We will then
apply Lemma~\ref{lemma:mutiple orthogonal sections} to get
$$|\mathcal{A}| \gtrsim \Haus{n''}(\eta(\tilde C))\approx
\Haus{n''}(\tilde C),$$ which gives \eqref{E:A bounded by tilde
C}. Hence we are let to construct  $C^1, \ldots, C^{b_0}, \tilde C_0$.

To construct $C^1, \ldots, C^{b_0}$, recall that $M'=M^1 \times \ldots \times M^{b_0}$,
and for every $i \in \{1, \ldots, b_0\}$ there exists $\bar y_i \in \partial^c \phi (x_0)$ with $\bar y_i^i =-x_0^i$. Moreover $M_{\bar y_i}^i \subset \partial^c \phi(x_0)$ by Lemma~\ref{L:phi on M i x 0}.
We further observe that the same inclusion holds for all the components $y_i^l$ of $y_i$ that satisfy $y_i^l=-x_0^l$. Hence, once $\bar y_i$ has a cut-locus component with $x_0$, then one can change such component arbitrarily, and the resulting point still remains inside $\partial^c \phi (x_0)$. 
Combining this fact with Loeper's maximum principle ({\bf DASM}) we can find a covector $v_i $ and a set $C^i \subset
\partial \phi (x_0) \subset T^*_{x_0} M$, with $v_i \in C^i$  whose components are either 
 $v_i^l = 0$ or $v_i^l \in M^{l*}(x_0^l)$, and
$$
C^i =\{ q \in T^*_{x_0} M \ | \ 2 q^i \in M^{i*}(x_0^i) \hbox{,
and } q^l =v_i^l \hbox{ for $l\ne i$ }\}.
$$
Clearly, $C^i$ is compact and $C^i \subset M^*(x_0)$. Moreover
$\cexp_{x_0} C^i \subset \partial^c \phi (x_0) \subset [\partial^c
\phi (x_0)]_{x_0} $ and $\Haus{n_i}(C^i) \gtrsim 1$. Also, observe
that the construction of $C^1, \ldots, C^{b_0}$ is
independent of $\e, \d, h$.

Let us now construct the set $\tilde C_0$. From Propositions~\ref{P:d
i h i} and \ref{P:regular in ''} we see that for $\d$ and $\frac{h}{\d}$
sufficiently small there exists a compact set $C'' \subset
\bar M''(x_0'')$ (independent of $\e,\d,h$) with $\pi''(\partial^c
\phi(Z_{\e,\d,h})) \subset C''$.  Recall the definition of $a_0$,
$b_0$, $\bar x_{\e,\d}=(\bar x_\e^\cd, \bar x_\d^\cdd)$, and that
$\bar x_\d^\cdd \in \bar M^\cdd(x_0^\cdd)$. Then we write
$\eta_{\e,\d}=(\eta_\e^\cd, \eta_\d^\cdd) \in T^*_{x_0^\cd}M^\cd
\times T_{x_0^\cdd}^* M^\cdd$ and we observe that $\eta_\d^\cdd$
is uniformly away from the boundary of $M^{\cdd *}
(x_0^\cdd)$. These facts imply that there exists a compact set
$C_2^\cdd \subset M^{\cdd*}(x_0^\cdd)$ (independent of $\e,\d,h$)
such that
$$ \eta (\tilde C) \subset \{\eta_{\e}^\cd\}\times C_2^\cdd.
$$
 However, $\eta_\e^\cd \to \partial M^{\cd*} (x_0^\cd)$ as $\e \to 0$, thus $\eta(\tilde C)$ is not kept in a fixed compact set in $M^*(x_0)$. In particular, we cannot take $\eta (\tilde C)$ for $\tilde C_0$, and this motivates the following:
Since $\bar x_\e^\cd = -x_0^\cd$ and $\bar x_{\e,\d} \in
[\partial^c\phi(Z_{\e,\d,h})]_{x_0}$, applying Lemma~\ref{L:phi on
M i x 0} as in the previous paragraph we see that
 the set $M^\cd \times \{ \bar x_\d^\cdd\}$, in particular,
 $(x_0^\cd, \bar
x_\d^\cdd)$ belongs to $ [\partial^c\phi(Z_{\e,\d,h})]_{x_0}$. This point
$(x_0^\cd, \bar x_\d^\cdd)$ corresponds to the covector
$(0,\eta_\d^\cdd)$. Consider the cone $\co\big{(}(0, \eta_\d^\cdd)\, ,
\, \eta(\tilde C) \big{)}$, and define $\tilde C_0$ as
$$
\tilde C_0 := \co\big{(}(0, \eta_\d^\cdd)\, , \, \eta(\tilde C)
\big{)} \cap \Big\{ \Big( \frac{\eta_\e^\cd}{2}, q^\cdd \Big) \in
T^*_{x_0}M \ | \ q^\cdd \in T^*_{x_0^\cdd}M^\cdd\Big\}.
$$
By a simple geometric argument
$$
\Haus{n''}(\tilde C_0) \gtrsim \Haus{n''}(\eta(\tilde C)),
$$
and moreover, since $\frac{\eta_\e^\cd}{2} \in  \frac{1}{2}M^{\cd*}(x_0^\cd)$,
the set $\tilde C_0$ is contained in a fixed
compact set  in $M^*(x_0)$  independently of $\e,\d,h$. By $c$-convexity of
$[\partial^c \phi (Z_{\e,\d,h})]_{x_0}$,
$$
\cexp_{x_0} (\tilde C_0) \subset [\partial^c \phi
(Z_{\e,\d,h})]_{x_0}.
$$
Note that by construction this set $\tilde C_0$, together with
$C^1, \ldots, C^{b_0}$, satisfy the property of the sets $S_i$ in
Lemma~\ref{lemma:mutiple orthogonal sections}. Furthermore they
are in a fixed compact set in $M^*(x_0)$ independent of $\e,\d,h$.
This completes the proof.
\end{proof}

Combining Propositions~\ref{P:bound partial c by tilde C} and
\ref{P:tilde C}(3) we obtain
\begin{align}\label{E:h n'' bound}
h^{n''} \lesssim \Haus{n''}(\pi''(\tilde Z_{\e,\d,h}))|\partial^c \phi (Z_{\e,\d,h})|.
\end{align}

We will finish the proof by applying Proposition~\ref{P:|Z h i|}. First, we need some preliminary steps.
 Use the notation given in Section~\ref{S:regular component}.
Let $Z_{\e,\d,h}'$ be the slice of $ Z_{\e,\d,h}$  in $M' \times \{ x''_0\}$, that is
$$
Z'_{\e, \d, h} := \{ x' \in M' \ | (x' , x''_0 ) \in Z_{\e,\d,h}  \}.
$$
Then $Z_{\e,\d,h}'$ is embedded via  $x' \mapsto -D_{\bar x'} c'(x',
\bar x'_{\e.\d})$ into $\tilde Z_{\e,\d,h}' \subset \bar M'^*(\bar
x_{\e,\d}')$, where
$$
\tilde Z'_{\e, \d, h} :=  \{ q' \in \bar M'^* (\bar x_{\e, \d}')  \ | \ (q', q_0'') \in \tilde
Z_{\e,\d,h}  \}.
$$
Embed in the same way each $Z_{\e,\d,h}^i$  (see \eqref{E:Z e d h i}), $i=1, \ldots b_0$,
into $\tilde Z_{\e,\d,h}^i  \subset \bar M^{i*} (\bar x_{\e,\d}^i).$
\begin{proposition}\label{P:Z h and tilde Z h'}
Assume that $0<\d\leq \d_0$ and $0<h\leq \d^2$,
with $\d_0$ as in Proposition~\ref{P:regular in ''}.
Then the following inequalities
hold:
$$
\Big(\min_{x' \in Z'_{\e,\d,h}} |\det (D_{x'}D_{\bar x'} c'(x',
\bar x_{\e,\d}'))| \Big)|Z'_{\e,\d,h}|  \le |\tilde Z'_{\e,\d,h} |
\le \Big(\max_{x' \in Z'_{\e,\d,h}} |\det (D_{x'}D_{\bar x'}
c'(x', \bar x_{\e,\d}'))| \Big)|Z'_{\e,\d,h}|,
$$
$$
\Big(\min_{x^i \in Z^i_{\e,\d,h}} |\det (D_{x^i}D_{\bar x^i}
c^i(x^i, \bar x_{\e,\d}^i))| \Big)|Z^i_{\e,\d,h}|  \le |\tilde
Z^i_{\e,\d,h} | \le \Big(\max_{x^i \in Z^i_{\e,\d,h}} |\det
(D_{x^i}D_{\bar x^i} c^i(x^i, \bar x_{\e,\d}^i))|
\Big)|Z^i_{\e,\d,h}|,
$$
where $|\cdot|$ denotes the Riemannian volume (in the appropriate submanifold).
\end{proposition}
\begin{proof}
From \eqref{E:q e}
$$
D_{x'} q' = - D_{x'}D_{\bar x'} c'( x'(q'), \bar x_{\e,\d}'),
$$
and so the first inequality follows from
$$
|\tilde Z_{\e,\d,h}' | = \int_{Z_{\e,\d,h}'} |\det D_{x'} q' |
\,dx' .
$$
The proof of the second inequality is analogous.
\end{proof}
By convexity  and Lemma \ref{lemma:mutiple orthogonal
sections} one has
$$
\Haus{n'}(\tilde Z_{\e,\d,h}')\gtrsim \prod_{i=1}^{b_0}
\Haus{n_i}(\tilde Z_{\e,\d,h}^i),
$$
while Propositions~\ref{P:Z h and tilde Z h'} and \ref{P:|Z h i|}
imply
\begin{align*}
\prod_{i=1}^{b_0} \Haus{n_i}(\tilde Z_{\e,\d,h}^i)
& \ge  \prod_{i=1}^{b_0}\Big(\min_{x^i \in Z^i_{\e,\d,h}} |\det (D_{x^i}D_{\bar x^i} c^i(x^i, \bar x_{\e,\d}^i))|\Big)| Z_{\e,\d,h}^i|\\
& \gtrsim \bigg{[}\prod_{i=1}^{b_0}\Big(\min_{x^i \in
Z^i_{\e,\d,h}} |\det (D_{x^i}D_{\bar x^i} c^i(x^i, \bar
x_{\e,\d}^i))|\Big)\bigg{]} \frac{h^{n'}}{\e^{a_0}}
\end{align*}
Combining these estimates with \eqref{E:h n'' bound} we get
\begin{align*}
h^{n'+ n''} & \lesssim \e^{a_0}
\bigg{[}\prod_{i=1}^{b_0}\Big(\min_{x^i \in Z^i_{\e,\d,h}} |\det
(D_{x^i}D_{\bar x^i} c^i(x^i, \bar
x_{\e,\d}^i))|\Big)\bigg{]}^{-1}
 \Haus{n'}(\tilde Z_{\e,\d,h}') \, \Haus{n''}(\pi''(\tilde Z_{\e,\d,h}))|\partial^c \phi (Z_{\e,\d,h})|\\
& \lesssim \e^{a_0} \bigg{[}\prod_{i=1}^{b_0}\Big(\min_{x^i \in
Z^i_{\e,\d,h}} |\det (D_{x^i}D_{\bar x^i} c^i(x^i, \bar
x_{\e,\d}^i))|\Big)\bigg{]}^{-1}
|\tilde Z_{\e,\d,h}||\partial^c \phi (Z_{\e,\d,h})|
 \qquad \hbox{(by Lemma~\ref{lemma:orthogonal sections})}\\
& \lesssim \e^{a_0} |Z_{\e,\d,h}| |\partial^c \phi (Z_{\e,\d,h})|,
\end{align*}
where the last inequality follows from
$$
|\tilde Z_{\e,\d,h}| \lesssim  \Big(\max_{x \in Z_{\e,\d,h}} \det
(D_{x} D_{\bar x} c(x, \bar x_{\e,\d})) \Big) |Z_{\e,\d,h}|
$$
(see Proposition~\ref{P:Z h and tilde Z h'}) and
$$
\frac{\max_{x' \in Z'_{\e,\d,h}} \det (D_{x'}D_{\bar x'} c'(x',
\bar x_{\e,\d}')) }{\Big{[}\prod_{i=1}^{b_0}\Big(\min_{x^i \in
Z^i_{\e,\d,h}} |\det (D_{x^i}D_{\bar x^i} c^i(x^i, \bar
x_{\e,\d}^i)|\Big)\Big{]}} \lesssim 1\qquad \mbox{as $\d,\frac{h}{\d} \to 0$}
$$
(see Propositions~\ref{P:d i h i} and
\ref{P:regular in ''}). This concludes the proof of
Theorem~\ref{T:right Alex}, and Theorem~\ref{T:stay-away} is
proved.

\bibliographystyle{plain}

\begin{thebibliography}{BFKM}




\bibitem[Br]{B} Y. Brenier, \emph{Polar factorization and monotone rearrangement of vector-valued functions.} Comm. Pure Appl. Math., 44 (1991), 375-417.


\bibitem[Ca1]{caff-loc} L. A. Caffarelli,
{\em A localization property of viscosity solutions to the
Monge-Amp\`ere equation and their strict convexity.} Ann. of
Math., 131 (1990), 129-134.

\bibitem[Ca2]{caffC1a} L. A. Caffarelli,
{\em Some regularity properties of solutions of Monge Amp\`ere
equation.} Comm. Pure Appl. Math., 44  (1991), no. 8-9, 965-969.


\bibitem[Ca3]{C} L. A. Caffarelli, \emph{The regularity of mapping with a convex potential.} J. Amer. Math. Soc., 5 (1992), 99-104.

\bibitem[Ca4]{C2} L. A. Caffarelli, \emph{Boundary regularity of maps with convex potentials II.} Ann. of Math., 144 (1996), 453-496.


\bibitem[Ca5]{caffGauss}
L. A. Caffarelli, \emph{Monotonicity properties of optimal transportation and the FKG and related inequalities.} Comm. Math. Phys., 214  (2000),  no. 3, 547-563.






\bibitem[Co]{Co} D. Cordero-Erausquin, \emph{Sur le transport de mesures p\'eriodiques.} C. R. Acad. Sci. Paris S\`er. I Math., 329 (1999), 199-202.

\bibitem[D1]{D} P. Delano\"e, \emph{Classical solvability in demension two of the second boundary value problem associated with the Monge-Amp\`ere operator.}
Ann. Inst. Henri Poincar\'e-Anal. Non Lin., 8 (1991), 443-457.

\bibitem[D2]{D2} P. Delano\"e, \emph{Gradient rearrangement for diffeomorphisms of a compact manifold.}
Diff. Geom. Appl., 20 (2004), 145-165.


\bibitem[DG]{DG} P. Delano\"e and Y. Ge, \emph{Regularity of optimal transportation maps on compact, locally nearly spherical, manifolds.}
To appear in J. Reine Angew. Math.

\bibitem[DL]{DL} P. Delano\"e and G. Loeper, \emph{Gradient estimates for potentials of invertible gradient mappings on the sphere.}
Calc. Var. Partial Differential Equations, 26 (2006), no. 3,
297-311.


\bibitem[FF]{FF} A. Fathi and A. Figalli, \emph{Optimal transportation on non-compact
manifolds.} To appear in Israel J. Math.

\bibitem[F]{F} A. Figalli, \emph{Existence, uniqueness, and regularity of optimal transport maps.}  SIAM J. Math. Anal., 39  (2007),  no. 1, 126-137.

\bibitem[FKM1]{FKM} A. Figalli, Y.-H. Kim and R. J. McCann, \emph{Continuity and injectivity of optimal maps\\ for non-negatively
cross-curved costs.} Preprint, 2009.

\bibitem[FKM2]{FigalliKimMcCann-econ09p}
A.~Figalli, Y.-H.~Kim and R.J.~McCann.
\newblock When is a multidimensional screening a convex program?
Uniqueness and stability of optimal strategies in the principal-agent problem.
\newblock Preprint at {\tt www.math.toronto.edu/mccann}.


\bibitem[FR]{FR} A. Figalli and L. Rifford,
\emph{Continuity of optimal transport maps on small deformations
of $\S^2$.} Comm. Pure Appl. Math., 62 (2009), no. 12, 1670-1706.


\bibitem[FRV]{FRV-reg} A. Figalli, L. Rifford and C. Villani,
\emph{Necessary and sufficient conditions
for continuity of optimal transport maps
on Riemannian manifolds.}
In preparation.


\bibitem[FV]{FV} A. Figalli and C. Villani,
\emph{An approximation lemma about the cut locus, with
applications in optimal transport theory.} Methods
Appl. Anal., 15 (2008), no. 2, 149-154.


\bibitem[G]{gutierrez} C. Gutierrez, \emph{The Monge-Amp\`ere Equation.} Birkhauser, 2001.


\bibitem[J]{John} F. John, \emph{Extremum problems with inequalities as subsidiary conditions.}
Studies and Essays Presented to R. Courant on his 60th Birthday,
January 8, 1948, 187-204. Interscience Publishers, Inc., New York,
N. Y., 1948.



\bibitem[K]{K} Y.-H. Kim, \emph{Counterexamples to continuity of optimal transportation on
positively curved Riemannian manifolds.}
 Int. Math. Res. Not. IMRN  2008, Art. ID rnn120, 15 pp.





\bibitem[KM1]{KM1} Y.-H. Kim and R. J. McCann, \emph{Continuity, curvature, and the general covariance of optimal transportation.}
Preprint at arXiv:0712.3077. To appear in J. Eur. Math. Soc.

\bibitem[KM1a]{KimMcCannAppendices} 
Y.-H. Kim and R.J. McCann. Appendices to original version of {\em Continuity, curvature, and the
  general covariance of optimal transportation}.
 Preprint at {arXiv:math/0712.3077v1}.

\bibitem[KM2]{KM2} Y.-H. Kim and R. J. McCann,
{\em Towards the smoothness of optimal maps on Riemannian submersions and Riemannian products
 (of round spheres in particular).}
To appear in J. Reine Angew. Math.


\bibitem[L]{Levin99}
V.L.~Levin.
\emph{Abstract cyclical monotonicity and {Monge} solutions for the general {Monge-Kantorovich} problem.}
Set-valued Anal. 7 (1999) 7--32.

\bibitem[L1]{L} G. Loeper, \emph{On the regularity of solutions of optimal transportation problems.}
Acta Math. 202  (2009),  no. 2, 241--283.


\bibitem[L2]{L2} G. Loeper, \emph{Regularity of optimal maps on the sphere: The quadratic cost and the reflector antenna.}
To appear in Arch. Ration. Mech. Anal.


\bibitem[LV]{LV} G. Loeper and C. Villani, \emph{Regularity of optimal transport in curved geometry: the nonfocal case.} To appear in Duke Math. J.


\bibitem[Li]{Liu09}
J.~Liu.
\newblock H{\"o}lder regularity of optimal mappings in optimal transportation.
\newblock {\em Calc Var. Partial Differential Equations} {\bf 34} (2009) 435--451.



\bibitem[M]{M} R. J. McCann, \emph{Polar factorization of maps on Riemannian manifolds.} Geom. Funct. Anal., 11 (2001), 589-608.

\bibitem[LTW]{LTW} J. Liu, N. Trudinger and X.-J. Wang, \emph{Interior $C^{2,\a}$-regularity for potential functions in optimal transportation},
To appear in Comm. Partial Differential Equations.

\bibitem[MTW]{MTW} X.-N. Ma, N. Trudinger and X.-J. Wang, \emph{Regularity of potential functions of the optimal transport problem.} Arch. Ration. Mech. Anal., 177(2): 151--183, 2005.


\bibitem[S]{S} T. Sei,
\emph{A Jacobian inequality for gradient maps on the sphere and its application to directional statistics.} Preprint at arXiv:0906.0874.




\bibitem[TW]{TW} N. Trudinger and X.-J. Wang,
\emph{On the second boundary value problem for Monge-Amp\`ere type
equations and optimal transportation.} Ann. Sc. Norm. Super. Pisa Cl. Sci. (5)  8  (2009),  no. 1, 143-174.



\bibitem[U]{U} J. Urbas, \emph{On the second boundary value problem for equations of Monge-Amp\`ere type.} J. Reine Angew. Math., 487 (1997), 115-124.


\bibitem[V]{V} C. Villani, \emph{Optimal Transport, Old and New}, Grundlehren des mathematischen Wissenschaften [Fundamental
Principles of Mathematical Sciences], Vol. 338, Springer-Verlag,
Berlin-New York, 2009.




\end{thebibliography}

\end{document}